\documentclass{amsart}
\usepackage{graphicx}
 \usepackage{stmaryrd}
\usepackage{color}
\def\a{\alpha}

\def\({\left (}
\def\){\right )}
\def\<{\left\langle}
\def\>{\right\rangle}

 \newtheorem{thm}{Theorem}[section]

\newtheorem{lem}[thm]{Lemma}
\newtheorem{prop}[thm]{Proposition}

\newtheorem{rem}[thm]{Remark}
\newtheorem{acknowledgement}{Acknowledgement}

\newcommand{\norm}[1]{\left\Vert#1\right\Vert}
\newcommand{\abs}[1]{\left\vert#1\right\vert}
\newcommand{\set}[1]{\left\{#1\right\}}
\newcommand{\Real}{\mathbb R}
\newcommand{\R}{\mathbb R}

\newcommand{\pfrac}[2]{\frac{\partial #1}{\partial #2}}

\begin{document}
\title{The  Yang-Mills  $\alpha$-flow  in vector bundles over four manifolds and its applications}

\numberwithin{equation}{section}
\author{Min-Chun Hong, Gang Tian and Hao Yin}

\address{Min-Chun Hong, Department of Mathematics, The University of Queensland\\
Brisbane, QLD 4072, Australia}  \email{hong@maths.uq.edu.au}

\address{Gang Tian, Department of Mathematics, Princeton University\\
USA}  \email{}

\address{Hao Yin,  School of Mathematical Sciences,
University of Science and Technology of China, Hefei, China}

\begin{abstract}
In this paper we introduce an $\alpha$-flow for the Yang-Mills
functional in vector bundles over four dimensional Riemannian
manifolds, and establish global existence of a unique smooth solution to
the $\alpha$-flow with smooth initial value.  We prove that the
limit of the solutions of the $\alpha$-flow  as $\alpha\to 1$ is a
weak solution to the Yang-Mills flow. By an application of the
$\alpha$-flow, we then follow the idea of Sacks and Uhlenbeck  \cite{SU} to prove some existence results for Yang-Mills connections and  improve the minimizing result of the Yang-Mills functional of   Sedlacek \cite{Se}.
\end{abstract}
\subjclass{AMS 58E15} \keywords{Yang-Mills flow, Sacks-Uhlenbeck
functional}

 \maketitle

 \pagestyle{myheadings} \markright {The Yang-Mills $\alpha$-flow}

\section{Introduction}
\label{sec:intro}

Suppose that $M$ is a connected compact four dimensional Riemannian manifold
and   $E$ is  a
 vector bundle over $M$.
 For each connection $D_A$, the Yang-Mills
 functional is defined by
 \begin{equation*}\mbox{YM} (A; M) =  \int_M |F_A|^2 \,dv, \end{equation*}
where  $F_A$ is the curvature of $D_A$. In a local trivialization,
we can express $D_A$ as $d+A$, where $A\in \Gamma (\text{End}
E\otimes T^*M)$ is the connection matrix.

We say that a connection $D_A$    is a {\it Yang-Mills connection}
if it is a critical point of the Yang-Mills functional; i.e. $D_A$
satisfies the Yang-Mills equation
 \begin{equation}\label{1.1}D_A^*F_A =0\,.  \end{equation}

    Yang-Mills equations originated from the theory of classical
fields in particle physics.  It turns out that  Yang-Mills theory has substantial applications in pure mathematics, especially in dimension 4.
In \cite{AHDM},  Atiyah,  Hitchin,
Drinfel'd   and   Manin established the fundamental existence
result of instantons on $S^4$.   Uhlenbeck \cite {remove}-\cite {U2}   established
 important analytic theorems for Yang-Mills
connections on $4$-manifolds. Donaldson  \cite{Do1}
 successfully applied the Yang-Mills theory to four dimensional geometric
topology.
\medskip

  The Yang-Mills equation is a typical example of partial differential equations involving gauge invariant of a group action. Besides its applications to geometry and topology, the  study of the existence of Yang-Mills connections is very interesting in itself.   Motivated by the seminal work of Eells-Sampson \cite {ES} on
harmonic maps, Atiyah and Bott \cite {AB} suggested to use
 the method of the Yang-Mills flow   to  establish the existence of
Yang-Mills connections.  The Yang-Mills flow equation is
 \begin{equation}\label{1.2} \pfrac{D_A}{t}=-D_A^*F_A,  \end{equation}
 with initial condition $D_A(0)=D_0$, where $D_0$ is a given smooth
 connection on $E$.    In
\cite {Do2},  Donaldson  used the Yang-Mills  flow to
establish   the important  result that an irreducible holomorphic
vector bundle $E$ over a compact K\"ahler surface $X$ admits a
unique Hermitian-Einstein connection if and only if it is stable.
Without the holomorphic structure of the bundle $E$,   it is still open
whether the Yang-Mills flow in four dimensional manifolds develop
a singularity in finite time.  Struwe \cite {St3} proved the
existence of  the weak solution to the Yang-Mills flow in vector
bundles on four manifolds, where
the weak solution is regular away from finitely many singularities in
$M\times (0,\infty)$. If
the Yang-Mills flow blows up at a finite
time $T>0$, the weak solution constructed by Struwe \cite {St3}
after the time $T$ lies on the new vector bundle $\tilde E$, which
might have  different second Chern number from the original
bundle $E$.

\medskip

 The Yang-Mills functional in dimension four is conformally invariant, which is similar to the conformal invariance of the Dirichlet energy of maps in dimension two, so there are general expectations that those results, which hold for harmonic maps from surfaces, should remain true in some sense for Yang-Mills connections in dimension four, if the gauge invariance problem is treated properly.
 In their celebrated paper \cite{SU}, Sacks and Uhlenbeck proposed to study the perturbed energy of a map $u$ from $M$ to $N$
\begin{equation*}
  E_\alpha(u)=\int_M (1+\abs{du}^2)^\alpha dv.
\end{equation*}
For $\alpha>1$, the functional $ E_\alpha(u)$ satisfies the Palais-Smale condition and therefore it is not difficult to find critical points of $E_\alpha$. They then analyzed the limit of the critical points when $\alpha$ goes to $1$. In spite of the possible blow-up phenomena, several interesting applications concerning the existence of harmonic maps were made. One of the major goals of this paper is to develop a parallel theory for the Yang-Mills functional in dimension four. Namely, we introduce the Yang-Mills $\alpha$-functional
\begin{equation*}
  YM_\alpha(A)=\int_M (1+\abs{F_A}^2)^\alpha dv.
\end{equation*}
The Euler-Lagrange equation for the functional $YM_\alpha$ is
\begin{equation}\label{eqn:alpha}
  D_A^*\left( (1+\abs{F_A}^2)^{\alpha-1}F_A \right)=0.
\end{equation}
A solution to  the Yang-Mills $\alpha$-equation (1.3) is called a
Yang-Mills $\alpha$-connection. In order to show the existence of
smooth $\alpha$-connections, one maybe check the Palais-Smale
condition for $YM_\alpha$ and then prove the regularity of the
weak solution of (\ref{eqn:alpha}). Instead, in this paper we
introduce the Yang-Mills $\alpha$-flow
\begin{equation}\label{eqn:alphaflow}
  \pfrac{A}{t}=-D_A^* F_A+(\alpha-1)\frac{*(d\abs{F_A}^2\wedge *F_A)}{1+\abs{F_A}^2}
\end{equation}
with initial condition $A(0)=A_0$.
Then we apply the Yang-Mills $\alpha$-flow to deform any given connection to a smooth Yang-Mills $\alpha$-connection.
More precisely, we prove
\begin{thm}\label{thm:one}
For a given smooth connection $A_0$, there exists a unique global
smooth solution $A_{\alpha }(x,t)$ to the evolution problem
(\ref{eqn:alphaflow}) in $M\times [0,\infty )$ for $\alpha-1$
sufficiently small. Moreover, for any $t_i\to \infty$, by passing to
a subsequence, $A_{\alpha}(\cdot,t_i)$ converges  up to
transformations
  to a limiting connection
$A^{\infty}_\alpha$ in $C^k(M)$ for any $k\geq 1$, and the
connection $A^{\infty}_\alpha$ is a smooth solution of (\ref{eqn:alpha}).
\end{thm}
To prove the global existence of the smooth solution of the
Yang-Mills $\alpha$-flow is not   easy   since the Yang-Mills
$\alpha$-flow is not parabolic. For the local existence of the
flow, we modify an idea of Donaldson \cite {Do2} to study
a equivalent flow. The main difficulty in
proving the global existence is how to extend the local solution to
 any time $T>0$. Due to the energy inequality, the Yang-Mills energy of the solution to the
 $\alpha$-flow does not concentrate at any time $T>0$ for each fixed $\alpha
>1$. However, we cannot  follow  the same proof of Struwe in
\cite{St2} to control the norm $H^{2}$ of the curvature $F$ since
the extra terms $\int_M|\nabla F|^4\,dv$ and $\int_M |F|^4\,dv$
come out due to the complexity of the $\alpha$-flow. Instead, we
  work on the  gauge-equivalent flow and prove that for any $t>0$, the Yang-Mills $\alpha$-flow has a smooth  solution in $M\times [t, t+t_0]$ for a fixed $t_0>0$, which depends on $YM_{\alpha}(A_{\alpha})$, so that we can extend the smooth solution to $M\times [0,\infty )$ (see Theorem 2.3).

\medskip

Following an idea from \cite{HY}, we apply  the Yang-Mills $\alpha$-flow to obtain a new proof of the existence  of a weak solution  of the Yang-Mills flow, which might be a different global weak solution from the one obtained by Struwe in \cite {St3}, as in the following.
\begin{thm}
    \label{thm:two} Let  $A_\alpha$ be the smooth solution of the  Yang-Mills
    $\alpha$-flow with the same initial condition $A_0$ for each $\alpha>1$. Then, there is a closed singularity set $\Sigma\subset M\times [0,\infty)$ with finite $2$-dimensional parabolic Hausdorff measure such that $\Sigma_t=\Sigma\cap (M\times \{t\})$ is at most a finite set for any $t$. There is a smooth bundle $\tilde{E}$ over $M\times [0,\infty)\setminus \Sigma$ with $\tilde{E}|_{M\times \{0\}}$ isomorphic to $E$ and a smooth connection $A_\infty(t)$ on $\tilde{E}|_{M\times \{t\}\setminus \Sigma_t}$ such that (1) $A_\infty(t)$ is a solution of the Yang-Mills flow; (2) for each compact set $K\subset M\times [0,\infty)\setminus \Sigma$, there are gauge transformations $\phi_\alpha$ over $K$ with $\phi_\alpha^*A_{\a}$ converging smoothly to $A_{\infty}$ over $K$ as $\a\to 1$.
\end{thm}

To prove Theorem \ref{thm:two}, we
establish a Bochner type estimate uniformly in $\alpha$ and a
local parabolic monotonicity formula for the Yang-Mills $\alpha$-flow,
which is similar to one in  \cite {St2} and  \cite {HT1}. Then we follow an idea of Schoen \cite
{Sch} (also see \cite {St2})  to obtain a uniform estimate on
$|F_{A_\alpha}|$ in $\alpha$. However,  there is a technical difficulty  that we do not have Bochner formulas for higher order derivatives of $F_{A_\alpha}$, so we cannot apply the Moser estimate to  obtain the unform estimates of higher order derivatives of $F_{A_\alpha}$.  To overcome this difficulty,  we obtain the uniform Sobolev norms of $\nabla^k_{A_{\alpha}}F_{A_{\alpha}}$ for all integers $k\geq 1$ by using the equation of $F_{A_{\a}}$ (see Lemma 3.6).

\medskip

With the analytic tools developed in the proof of the previous two theorems, we investigate further applications of the $\alpha$-flow. It is not hard to establish an $\varepsilon$-regularity result for studying the blow-up of a sequence of Yang-Mills  $\alpha$-connections.  When a blow-up phenomenon happens, we will study the change of the topology of the bundle. More precisely, the original bundle $E$, on which the blow-up sequence lies, is the connected sum of the weak limit bundle over $M$ and the bubbling bundles over $S^4$.  Following the idea  of Sacks and Uhlenbeck's paper \cite{SU},  we  apply  the existence of smooth  Yang-Mills  $\alpha$-connections of Theorem \ref{thm:one} to show

\begin{thm}
    \label{thm:four}
    If $\pi_3(G)$ is a free abelian group of rank $r$, then there exist  at least $r$ different Yang-Mills $G$-connections over $S^4$.
\end{thm}

\begin{rem}
    It is well known that any simple compact Lie group $G$ has $\pi_3(G)=\mathbb Z$. So the result is useful only for semi-simple compact Lie groups, for example $SO(4)$.
\end{rem}

Furthermore, we can apply the Yang-Mills $\alpha$-flow  to improve the minimizing theory of the Yang-Mills functional on $E$. In \cite{Se}, Sedlacek studied the direct minimizing method for the Yang-Mills functional in $E$.      More precisely, let $D_i$ be a minimizing sequence in the
given bundle $E$ over $M$. Using the weak compactness result of Uhlenbeck \cite
{U2}, Sedlacek  proved that $D_{i}$ weakly converges in
$W^{1,2}(M\backslash \{x_1,...x_l\})$ to a limiting connection $D_{\infty}$ which can be extended to  a Yang-Mills connection  in
a (possibly) new bundle $E'$  over $M$ with the same topological  invariant $\eta (E')=\eta (E)$, which is an element of $H^2(M,\pi_1(G))$.
Because there is only $W^{2,2}$ control of the transition functions, one can not use the gluing argument of Uhlenbeck in \cite{U2} to obtain a bundle map. Therefore, the relation between the original bundle and the limit bundle $E'$ (which may be different) is not quite clear. It is known that the topology of a vector bundle over a $4$-manifold is determined by some $\eta$ invariant, and the vector Pontryagin number (see the appendix in \cite{Se}).
By using the $\alpha$-flow, we modify the minimizing sequence to obtain a better control and  new minimizing sequence, which converges to the same limit in the smooth topology up to gauge transformation away from finite singular points. Moreover, for the modified minimizing sequence, a blow-up analysis is discussed and an energy identity is proved.

\begin{thm}\label{thm:three}
    Let  $E$ be a vector bundle over $M$ with structure group $G$. Assume that $D_i$ is a minimizing sequence of the Yang-Mills functional $YM$ among smooth connections on $E$, which converges weakly to some limit connection $D_{\infty}$ by Sedlacek's result. There is a modified minimizing sequence ${D'_i}$, a finite set $S\subset M$ and a sequence of gauge transformations $\phi_i$ defined on $M\setminus S$, such that for any compact $K\subset M\setminus S$, $\phi_i^*{D'}_i$ converges to $D'_{\infty}$ smoothly in $K$, where $D'_{\infty}$ is gauge equivalent to the connection $D_{\infty}$.
Moreover, there are a finite number of bubble bundles $E_1,\cdots,E_l$ over $S^4$ and Yang-Mills connections $\tilde{D_1},\cdots,\tilde{D}_l$ such that
    \begin{equation*}
        \lim_{i\to\infty} YM(D_i)=YM(D_{\infty})+\sum_{j=1}^l YM(\tilde{D}_j).
    \end{equation*}
\end{thm}
This improves Theorem 5.5 of \cite{Se} because the convergence of $\phi_i^*{D'}_i$ is smooth.  (See \cite{Isobe} for a similar discussion using Sobolev bundles and the weak convergence.)

\medskip

Finally, we would like to discuss some potential application of the  Yang-Mills  $\alpha$-flow to the Morse theory of the Yang-Mills functional. It is well known that the Yang-Mills functional in dimension four does not satisfy the Palais-Smale condition. Many efforts have been made in this direction (see \cite{Taubes} and the references therein). Following an idea in \cite{SU}, one expects to study the limiting solutions of the   $\alpha$-equations (1.3) as $\alpha$ goes to $1$. It seems that the Yang-Mills $\alpha$-flow provides a new analytic tool to prove the  existence of Yang-Mills connections. In Subsection 4.4, we use it as the analytic tool to provide a new proof of the existence of the nonminimal Yang-Mills connection on $S^4$, which is due to Sibner, Sibner and Uhlenbeck \cite{SSU}.

\medskip

The rest of the paper is organized as follows: In Section \ref{sec:alphaflow}, we prove Theorem \ref{thm:one} and some other analytic results needed for the applications. In Section \ref{sec:convergence}, we study the limit of the $\alpha$-flow as $\alpha$ goes to $1$ and prove Theorem \ref{thm:two}. In the final section, we study serval applications of the $\a$-flow.

\section{Existence of the $\alpha$-flow and its equivalent flow}
\label{sec:alphaflow}
\subsection{Local existence of the $\alpha$-flow}
\label{sub:deturk}
It is well known that (\ref{eqn:alphaflow}) is not a parabolic system and that this difficulty can be overcome by using a kind of Deturk trick. Throughout this paper, let $D_{ref}$ be a fixed smooth background connection.

Let $D_0=D_{ref} +A_0$ be a given smooth connection in $E$.

 Following \cite{St2}, we consider an equivalent flow
\begin{equation}\label {eqn:modifiedflow}
    \pfrac{\bar D}{t}=-\bar D^*F_{\bar D}+(\alpha-1)\frac{*( d|F_{\bar D}|^2 \wedge *F_{\bar D})}
    {1+\abs{F_{\bar D}}^2}-\bar D(\bar D^*a),
\end{equation}
with $\bar D(t)=D_{ref}+a(t)$ and $a(0)=A_0$. Then the equivalent flow is a nonlinear parabolic
system. By the well-known theory of partial differential
equations, there is a unique smooth solution of (\ref{eqn:modifiedflow})
defined on $M\times[0,T]$ for some $T>0$. By the
theory of ordinary differential equations, there is a unique
solution to the following initial problem:
 \begin{equation}\label {eqn:ODE} \frac d{dt} S=-S\circ (\bar
 D^*a),
\end{equation}
 $M\times [0, T]$, with initial value  $S(0)=I$. Here $S(t)$ is a global gauge transformation and $I$ is the trivial one.

Setting
\begin{equation*}
  D=(S^{-1})^* \bar{D},
\end{equation*}
we have (e.g. see \cite {St2}, \cite {Ho})
\[F_{\bar D}=S^{-1} F S,\quad \bar D(\bar D^*a) = \bar D\circ (\bar D^*a)-\bar D^*a \circ \bar D.\]
Combining (\ref{eqn:modifiedflow}), (\ref{eqn:ODE}) with the above facts yields
  \begin{eqnarray*}
    \frac d{dt} D&=& \frac {dS}{dt} \circ \bar D \circ S^{-1}  +S\circ \frac {d \bar D}{dt} \circ  S^{-1} +S\circ \bar D \circ \frac {dS^{-1}}{dt}  \\
    &=&S\left ( -\bar D^*F_{\bar D}+(\alpha-1)\frac{*( d |F_{\bar D}|^2\wedge *F_{\bar D})} {1+\abs{F_{\bar D}}^2}\right )S^{-1}\\
    &=&-D_A^*F_A+(\alpha-1)\frac{*( d |F_A|^2 \wedge *
    F_A)}{1+\abs{F_A}^2}.
 \end{eqnarray*}
 This shows that
$D=(S^{-1})^* \bar D$ satisfies the Yang-Mills $\alpha$-flow with
$D(0)=D_0$ in $M\times [0,T]$ for some $T>0$.

Next, we remark that the smooth solution of the Yang-Mills
$\alpha$-flow is unique. In fact, let $D_i=D_{ref}+A_i (i=1,2)$ be two smooth
solutions to the Yang-Mills $\alpha$-flow with $A_i(0)=A_0$.
By the theory of parabolic equations, there is a unique local
smooth solution of the parabolic system of second order:
 \begin{eqnarray}
   \label {eqn:gaugeflow}
   \frac d{dt} S_i=- (D_{ref} +A_i)^*[ A_i  S_i + D_{ref} S_i]
 \end{eqnarray}
 with $S(0)=I$. By computation, we can check that the connections $\bar{D}_i=S^*(D_i)$ are two solutions to the modified flow (\ref{eqn:modifiedflow}) with the same initial value. Hence, $\bar{D}_1$ and $\bar{D}_2$ are the same. Moreover, (\ref{eqn:gaugeflow}) is nothing but the ODE (\ref{eqn:ODE}). By the uniqueness of ODEs, we know $S_i$ and hence $D_i$ are the same.

A similar method to
prove uniqueness was used for the Ricci flow and also for the
Seiberg--Witten flow \cite {HS}. Therefore, we have shown that the
$\alpha$-flow has a unique solution in $M\times [0,T)$ for some
$T>0$.

\subsection{Energy inequality of the $\alpha$-flow}

\begin{lem}  \label{lem:energyinequality} Let $A(t)$ be a solution to the Yang-Mills $\alpha$-flow in $M\times
 [0,T)$ with initial value $A(0)=A_0$. For each $0<t<T$, we have
\begin{equation} \label{Energy identity}
  \int_M (1+|F|^2 )^{\alpha} \,dv +
  2\alpha\int_0^t \int_M (1+\abs{F}^2)^{\alpha-1}\abs{\pfrac{A}{s}}^2 dv\,ds= \int_M (1+|F_{A_0}|^2 )^{\alpha}
  \,dv.
\end{equation}
\end{lem}
\begin{proof} Note $\frac {\partial F}{\partial t} =D\frac
{\partial A}{\partial t}$. Then,
    multiplying (\ref{eqn:alphaflow}) by $(1+\abs{F}^2)^{\alpha-1} \partial_t A$ and integrating by
    parts, we have
\begin{eqnarray*}
    \frac{d}{dt}\int_M (1+\abs{F}^2)^{\alpha} dv
    &=& 2\alpha \int_M \left <(1+\abs{F}^2)^{\alpha-1} F, \frac{\partial F}{\partial t}\right > dv\\
    &=& 2\alpha \int_M  \left <D^* ( (1+\abs{F}^2)^{\alpha-1}F),\frac{\partial A}{\partial t} \right >\,dv \\
    &=& -2\alpha \int_M (1+\abs{F}^2)^{\alpha-1}\abs{\pfrac{A}{t}}^2 dv.
\end{eqnarray*}
Then (\ref{Energy identity}) follows from integrating over $[0,t]$.
\end{proof}

\begin{lem}  \label{lem:localenergy} Let $A(t)$ be a solution to the Yang-Mills $\alpha$-flow in $M\times [0,T)$. For each $0<t_1<t_2<T$, we have
\begin{equation} \label{eqn:localenergy}
    \int_{B_R(x)} (1+\abs{F}^2)^\alpha (t_2) dv \leq \int_{B_{2R(x)}} (1+\abs{F}^2)^\alpha (t_1) dv + C\frac{t_2-t_1}{R^2} YM_0.
\end{equation}
Here $YM_0$ is an upper bound of the overall energy.
\end{lem}
\begin{proof}
    Let $\varphi$ be a cut-off function supported in $B_{2R}(x)$ and $\varphi\equiv 1$ on $B_{R}(x)$.
\begin{eqnarray*}
    \frac{d}{dt}\int_M \varphi^2 (1+\abs{F}^2)^{\alpha} dv
    &=& 2\alpha \int_M \varphi^2  \left <D^* ( (1+\abs{F}^2)^{\alpha-1}F),\frac{\partial A}{\partial t} \right > \\
    && + \varphi (1+\abs{F}^2)^{\alpha-1} F\# \nabla \varphi \# \pfrac{A}{t} \, dv \\
    &\leq &- \int_M \varphi^2 (1+\abs{F}^2)^\alpha \abs{\pfrac{A}{t}}^2 + (1+\abs{F}^2)^{\alpha-1} \abs{\nabla \varphi}^2 \abs{F}^2 dv.
\end{eqnarray*}
The lemma follows from integration over $[t_1,t_2]$.
\end{proof}

\subsection{Global existence of the $\alpha$-flow}

In this section, we will show that the solution of the Yang-Mills
$\alpha$-flow (for small $\alpha-1$) exists in $M\times [0,T)$ for all $T>0$.

\begin{thm}\label{thm:globalexist}
    Let $D_0=D_{ref} +A_0$ be a smooth connection in $E$. Then
      there is a smooth solution $A$ to the $\alpha$-flow
      (\ref{eqn:alphaflow}) with initial value $A_0$ in $M\times [0,t_0)$ for a constant
$t_0>0$ depending only on $YM_\alpha(D_0)$.
\end{thm}

We note that together with Lemma \ref{lem:energyinequality} and the uniqueness of smooth solution to (\ref{eqn:alphaflow}), Theorem \ref{thm:globalexist} implies the global existence part of Theorem \ref{thm:one}.

The proof involves higher order estimates for parabolic systems. For that purpose, we resort to the modified flow (\ref{eqn:modifiedflow}) again. To start the proof, we need the following lemma.
\begin{lem}
  \label{lem:byUhlenbeck}
  Let $D$ be a smooth connection on $E$ with $YM_\alpha(D)$ bounded, and let $D_{ref}$ be some fixed reference connection on $E$. Then there exists a global smooth gauge transformation $s$ such that
  \begin{equation*}
    \norm{s^* D-D_{ref}}_{W^{1,2\alpha}(M)}\leq C.
  \end{equation*}
  Here $C$ is some constant depending only on $D_{ref}$ and $YM_\alpha(D)$.
\end{lem}
\begin{proof}
    Although not explicitly stated, the proof is essentially contained in the paper \cite{U2} of Uhlenbeck. We briefly indicate how it follows from \cite{U2}.

    If the lemma is not true, then there exists a sequence of $D_i$ with  $YM_\alpha(D_i)$ uniformly bounded such that for any smooth gauge transformation $s_i$, we have
  \begin{equation}\label{eqn:contra}
    \norm{s_i^*D_i-D_{ref}}_{W^{1,2\a}(M)}\geq i.
  \end{equation}
  It is shown in \cite{U2} that by passing to some subsequence, there exists $s_i$ such that $s_i^*D_i$ converges weakly in $W^{1,p}$ to some $D_\infty$ for $p=2\alpha$.

  In the proof, Uhlenbeck chose some $j$ sufficiently large and wrote $s_i^*D_i$ in local trivialization $\sigma_\alpha(j)$ as
  \begin{equation*}
    d+\rho^{-1}_\alpha(i) d\rho_\alpha(i) +\rho^{-1}_\alpha(i) A(\alpha,i)\rho_\alpha(i).
  \end{equation*}
  Here we  refer the reader to \cite{U2} to see the definitions of $\sigma_\alpha(i)$, $\rho_\alpha(i)$ and $A(\alpha,i)$.
 Moreover, Uhlenbeck proved that
  \begin{equation*}
    \rho^{-1}_\alpha(i) d\rho_\alpha(i) +\rho^{-1}_\alpha(i) A(\alpha,i)\rho_\alpha(i)
  \end{equation*}
  is bounded in $W^{1,p}$ uniformly in  $i$. Although the local expression of $D_{ref}$ in the trivialization $\sigma_\alpha(j)$ has no explicit bound, it is independent of $i$. Hence $s^*_i D_i-D_{ref}$ is bounded in $W^{1,p}$ uniformly in $i$ locally in the trivialization $\sigma_\alpha(j)$. Since $s^*_i D_i-D_{ref}$ is a tensor and we may show the same bound in $\sigma_\beta(j)$ for $\beta\ne \alpha$. We get a contradiction with (\ref{eqn:contra}) and the lemma is proved.
\end{proof}

With this lemma, we may assume without loss of generality that $A_0$ in Theorem \ref{thm:globalexist} has bounded $W^{1,2\alpha}$ norm.

\begin{proof} [Proof of Theorem \ref{thm:globalexist}]
    Instead of (\ref{eqn:alphaflow}), we shall discuss (\ref{eqn:modifiedflow}). By our discussion in Subsection \ref{sub:deturk}, we know this is sufficient.

    For some $\varepsilon>0$ to be determined later, the H\"older inequality and Lemma \ref{lem:byUhlenbeck} imply that there exist  $r_0>0$ and $C_1>0$ such that for all $x\in M$,
  \begin{equation}
    \int_{B_{r_0}(x)} \abs{A_0}^2 +\abs{\nabla_{ref} A_0}^2 dx \leq \varepsilon/2
    \label{eqn:condition1}
  \end{equation}
  and
  \begin{equation}
    \int_M \abs{A_0}^2+\abs{\nabla_{ref} A_0}^2 dx \leq C_1.
    \label{eqn:condition2}
  \end{equation}
  Let $\{x_i\in M| i=1,\cdots,L\}$ be a finite number of points in $M$ such that $\{B_{r_0}(x_i)\}$ covers $M$ and for each $i$ there are at most $k$ different $j$'s ball $B_{r_0}(x_j)$ with $B_{2r_0}(x_i)\cap B_{r_0}(x_j)\ne \emptyset$. Although $L$ depends on $\varepsilon$, it is important to note that $k$ is a universal constant depending only on the dimension.

  Let $D(t)=D_{ref}+a(t)$ be the local solution to (\ref{eqn:modifiedflow}) defined on $[0,T)$. Since $a(t)$ is smooth, there exists a $t_1>0$  which is the maximal time in $[0,T]$ such that  for all $i=1,\cdots,L$,
  \begin{equation}
    \sup_{0\leq t<t_1}
    \int_{B_{r_0}(x_i)} \abs{a(t)}^2 +\abs{\nabla_{ref} a(t)}^2 dx \leq \varepsilon
    \label{eqn:condition3}
  \end{equation}
  and
  \begin{equation}
    \sup_{0\leq t<t_1} \int_M \abs{a(t)}^2+\abs{\nabla_{ref} a(t)}^2 dx + \int_0^{t_1}\int_M \abs{\nabla^2_{ref}a }^2 dxdt \leq 2C_1.
    \label{eqn:condition4}
  \end{equation}

  We shall find $t_0$ depending on $YM_\alpha(D_0)$ and $\alpha$ alone (the exact value of $t_0$ is determined in the process of proof) and prove that $T\geq t_0$, which concludes the proof of the theorem. If not, then either $t_1<T<t_0$ or $t_1=T<t_0$. It suffices to show that neither case is possible.

    Before we give the details of the proof, we outline the idea of the proof. By Lemma \ref{lem:byUhlenbeck}, we have (\ref{eqn:condition1}) and (\ref{eqn:condition2}) for the initial value $a(0)$.  Step 1 below shows that as long as the solution exists, (\ref{eqn:condition3}) and (\ref{eqn:condition4}) must remain true for $t\in [0,t_0]$ for some $t_0>0$ depending only on $YM_\alpha(D_0)$. The condition (\ref{eqn:condition3}) is a `smallness' condition, which will enable us to prove higher derivative estimates for the nonlinear parabolic system (\ref{a1}) of second order. This is done in  below Step 2.

  {\bf Step 1:} $t_1<T<t_0$ is not possible.

   To study the evolution of $a(t)$,   we rewrite the flow equation (\ref{eqn:modifiedflow}) as
    \begin{eqnarray}\label{a1}
       \quad  \frac{\partial a}{\partial t} &=& \triangle_{ref} a +(\nabla_{ref} a\# a +a\#a\#a)-D_{ref}^*F_{ref}\\
        &&+(\alpha-1) \psi(F_D)\# (\nabla_{ref}^2a
        +a\#\nabla_{ref}
        a+a\#a\#a + \nabla_{ref}F_{ref}),\nonumber
    \end{eqnarray}
    with the initial value $a(0)=A_0$,
    where $\psi (F_D)$ is a bounded function depending on $F_D$. For any $i$, let $\phi_i$ be a cut-off function supported in $B_{2r_0}(x_i)$ with $\phi_i\equiv 1$ on $B_{r_0}(x_i)$. For simplicity, we write $\phi$ when it applies to all $\phi_i$.

Multiplying $(\ref {a1})$ by $a$ and using Young's inequality, we
have
\begin{eqnarray}\label{a2}
        &&\frac{d}{dt} \int_M  |a|^2\,dv +\int_M  |\nabla_{ref}a|^2\,dv\\
        &\leq &\frac 1 2\int_M  {|\nabla_{ref}
        a|^2}\,dv +C(\alpha-1)\int_M |\nabla_{ref}^2 a|^2\,dv + C \int_M {|a|^4}\,dv+C. \nonumber
    \end{eqnarray}
    By our choice of $t_1$, we have for $t<t_1$,
    \begin{equation*}
      \frac{d}{dt}\int_M \abs{a}^2 dv +\frac{1}{2} \int_M \abs{\nabla_{ref} a}^2 dv \leq C(\alpha-1)\int_M \abs{\nabla_{ref}^2 a}^2 dv +C.
    \end{equation*}

Multiplying $(\ref {a1})$ by $\triangle_{ref} a$, we have
\begin{eqnarray}\label{31}
        &&\frac{d}{dt}\int_M |\nabla_{ref} a|^2\,dv +\int_M |\triangle_{ref} a|^2\,dv\\
        &\leq &\frac 1 2\int_M |\triangle_{ref}
        a|^2\,dv+ C(\alpha-1)\int_M|\nabla_{ref}^2 a
        |^2\,dv
       \nonumber \\
        &+& \int_M(|\nabla_{ref} a|^2 |a|^2 +|a|^6) \,dv+C.\nonumber
    \end{eqnarray}
    By H\"older's inequality and the Sobolev inequality, we obtain
    \begin{eqnarray*}
      \int_M \abs{a}^6 dv &\leq& \sum_i \int_{B_{r_0}(x_i)} \abs{a}^6 dv \\
      &\leq& \sum_i \left( \int_{B_{r_0}(x_i)} \abs{a}^4 \right)^{1/2} \left( \int_{B_{r_0}(x_i)} \abs{a}^8 dv \right)^{1/2} \\
      &\leq&  \varepsilon \sum_i  \int_{B_{r_0}(x_i)}\abs{\nabla_{ref} a}^2 \abs{a}^2 + \abs{a}^4 dv \\
      &\leq& C \varepsilon \int_M \abs{\nabla_{ref} a}^2 \abs{a}^2 + \abs{a}^4 dv \\
      &\leq& C \varepsilon \int_M \abs{\nabla_{ref} a}^2 \abs{a}^2 dv + C. \\
    \end{eqnarray*}
    Similarly,
    \begin{eqnarray*}
      \int_M \abs{\nabla_{ref} a}^2\abs{a}^2 dv &\leq& \sum_i \int_{B_{r_0}(x_i)} \abs{\nabla_{ref} a}^2\abs{a}^2 dv \\
      &\leq& \sum_i \left( \int_{B_{r_0}(x_i)} \abs{a}^4 \right)^{1/2} \left( \int_{B_{r_0}(x_i)} \abs{\nabla_{ref}a}^4 dv \right)^{1/2} \\
      &\leq& \varepsilon \sum_i \int_{B_{r_0}(x_i)} \abs{\nabla^2_{ref}a}^2+ \abs{\nabla_{ref} a}^2 dv  \\
      &\leq& C \varepsilon \int_M \abs{\nabla^2_{ref}a}^2+ \abs{\nabla_{ref} a}^2 dv  \\
    \end{eqnarray*}
    Using integration by parts, we have
    \begin{equation*}
      \int_M \abs{\nabla^2_{ref} a}^2 dv \leq \int_M \abs{\triangle_{ref} a}^2 dv + C\int_M \abs{\nabla_{ref} a}^2 dv,
    \end{equation*}
    which implies
    \begin{equation*}
      \frac{3}{4}\int_M \abs{\nabla^2_{ref} a}^2 dv \leq \int_M \abs{\triangle_{ref} a}^2 dv + C.
    \end{equation*}
    In summary, by choosing $\alpha-1$ and $\varepsilon$ small, we have
    \begin{equation*}
      \frac{d}{dt} \int_M \abs{a}^2 +\abs{\nabla_{ref} a}^2 dv +\frac{1}{4} \int_M \abs{\nabla_{ref} a}^2 +\abs{\nabla^2_{ref} a}^2 dv\leq C
    \end{equation*}
    for $t\in [0,t_1]$.
    Integrating the above inequality yields that there exists $t_0>0$ such that (\ref{eqn:condition4}) remains true for $t_1\leq t_0$.

    For (\ref{eqn:condition3}), we need a local version of the above computation.
Multiplying $(\ref {a1})$ by $\phi_i^2 a$ and using Young's inequality, we
have
\begin{eqnarray}\label{locala2}
  &&\frac{d}{dt} \int_M  |a|^2\phi_i^2 \,dv +\frac{1}{2}\int_M  |\nabla_{ref}a|^2\phi_i^2 \,dv\\
        &\leq &C(\alpha-1)\int_M |\nabla_{ref}^2 a|^2\phi_i^2 \,dv +C. \nonumber
\end{eqnarray}
Here we have used the bound on $\abs{\nabla \phi_i}$ and $\int_M \abs{a}^4 dv$ for $t\leq t_1$.
Multiplying $(\ref {a1})$ by $\phi_i^2 \triangle_{ref} a$, we have
\begin{eqnarray}\label{local31}
  &&\frac{d}{dt}\int_M |\nabla_{ref} a|^2 \phi^2_i\,dv +\frac{1}{2}\int_M |\triangle_{ref} a|^2\phi_i^2 \,dv\\ \nonumber
        &\leq & C(\alpha-1)\int_M|\nabla_{ref}^2 a|^2 \phi_i^2\,dv
        + \int_M(|\nabla_{ref} a|^2 |a|^2\phi_i^2  +|a|^6\phi_i^2) \,dv+C\\\nonumber
    &&+ C\int_M \abs{\nabla_{ref} a }^2 \abs{\nabla \phi_i}^2 dv.
    \end{eqnarray}
    By integration by parts, we have
    \begin{eqnarray*}
      \frac{3}{4}\int_M \abs{\nabla^2_{ref} a}^2\phi_i^2 dv& \leq& \int_M \abs{\triangle_{ref} a}^2\phi_i^2 dv + C\int_M \abs{\nabla_{ref} a}^2(\phi_i^2+\abs{\nabla \phi_i}^2) dv \\
      &\leq& \int_M \abs{\triangle_{ref} a}^2 \phi^2_i dv + C,
    \end{eqnarray*}
    where we have used (\ref{eqn:condition4}) for $t<t_1$.

    We can deal with the main nonlinear terms as before.
    \begin{eqnarray*}
      \int_M \abs{a}^6 \phi_i^2 dv &\leq& C\varepsilon \int_M \abs{\nabla_{ref} (\varphi a^2)}^2 +\varphi^2 \abs{a}^4 dv \\
      &\leq& C\varepsilon \int_M (\abs{\nabla \phi_i}^2 +\phi_i^2) \abs{a}^4 + \phi_i^2 \abs{a}^2 \abs{\nabla_{ref} a}^2 dv \\
      &\leq& C\varepsilon \int_M \phi_i^2 \abs{a}^2 +\abs{\nabla_{ref} a}^2 dv +C
    \end{eqnarray*}
    and
    \begin{eqnarray*}
      \int_M \phi_i^2 \abs{a}^2 \abs{\nabla_{ref} a}^2 dv &\leq& C\varepsilon \int_M \abs{\nabla_{ref} (\phi_i \nabla_{ref} a)}^2 + \phi_i^2 \abs{\nabla_{ref} a}^2 dv \\
      &\leq& C\varepsilon \int_M \phi_i^2 \abs{\nabla_{ref}^2 a}^2 dv +C.
    \end{eqnarray*}
In summary, for $t<t_1$, we have
\begin{equation*}
  \frac{d}{dt} \int_M \phi_i^2( \abs{a}^2 +\abs{\nabla_{ref}a}^2 ) dv \leq C.
\end{equation*}
Therefore, by choosing $t_0$ sufficiently small, we see that both (\ref{eqn:condition3}) and (\ref{eqn:condition4}) remain true for $t_1\leq t_0$. By our definition of $t_1$, this shows $t_1<T<t_0$ is not possible.

{\bf Step 2:} $t_1=T<t_0$ is not possible.

As pointed out before in Step 1, we now show thtat (\ref{eqn:condition3}) and (\ref{eqn:condition4}) together with (\ref{a1}) imply higher order estimates up to $T$, so that the solution can be extended beyond $T$.

For that purpose, we consider the evolution equation of $a$. Let $\varphi$ be a cut-off function in time. Precisely, $\varphi(t)\equiv 0$ for $t<t_1/4$ and $\varphi(t)\equiv 1$ for $t\in [t_1/4,t_1]$. Multiplying (\ref{a1}) with $\varphi^3$ and applying the $L^p$ estimate, we obtain for $p=4$,
\begin{eqnarray*}
  \norm{\varphi^3 a}_{W^{2,1}_p(M\times [0,t_1])}&\leq& C (\alpha-1) \norm{\varphi^3 \nabla^2_{ref} a}_{L^p(M\times [0,t_1])} + C\norm{\varphi^3\nabla_{ref}a\# a}_{L^p(M\times [0,t_1])} \\
  &&+ C\norm{\varphi^3 a\#a\#a}_{L^p(M\times [0,t_1])} +C.
\end{eqnarray*}
We denote $W^{2,1}_p$  by the  space of functions whose  space derivatives up to second order and first order time derivative belong to $L^p$.
 The $L^p$ norm of $\varphi^2 \partial_t \varphi a$ is bounded by (\ref{eqn:condition4}), which is why we assume $p=4$.

By choosing $\alpha-1$ sufficiently small and using Young's inequality, we have
\begin{equation*}
  \norm{\varphi^3 a}_{W^{2,1}_p(M\times [0,t_1])}\leq C \norm{ \varphi a}^{3}_{L^{3p}(M\times [0,t_1])} + C\norm{\varphi^{2}\nabla_{ref} a}_{L^{3p/2}(M\times [0,t_1])}^{3/2}+C.
\end{equation*}

Recall that $M$ is covered by $B_{r_0}(x_i)$ and $\int_{B_{r_0}(x_i)} \abs{a}^4 dv\leq C\varepsilon^2$. For simplicity, we write $B_i$ for $B_{r_0}(x_i)$.
An interpolation theorem of Nirenberg (Theorem 1 in \cite{Nirenberg}) implies that
\begin{equation*}
  \norm{\varphi a}_{L^{3p}(B_i)} \leq C \norm{\varphi^3 \nabla_{ref}^2 a}_{L^p(B_i)}^{1/3}\norm{a}_{L^4(B_i)}^{2/3} +C \norm{a}_{L^4(B_i)}.
\end{equation*}
This implies that
\begin{equation*}
  \int_{B_i} \abs{\varphi a }^{3p} dv \leq C \varepsilon^{p} \int_{B_i} \abs{\varphi^3 \nabla_{ref}^2 a}^p dv + C.
\end{equation*}
Hence,
\begin{eqnarray*}
  \int_o^{t_1}\int_M \abs{\varphi a}^{3p} dv &\leq&\int_0^{t_1}\sum_i \int_{B_{r_0}(x_i)} \abs{\varphi a}^{3p} dv \\
  &\leq& C\varepsilon^p \int_0^{t_1} \int_M \abs{\varphi^3 \nabla_{ref}^2 a}^p dv +C.
\end{eqnarray*}
That is
\begin{equation*}
  \norm{\varphi a}_{L^{3p}(M\times [0,t_1])}^3\leq C\varepsilon \norm{\nabla_{ref}^2 (\varphi^3 a)}_{L^p(M\times [0,t_1])} +C.
\end{equation*}
Similarly,
\begin{equation*}
  \norm{\varphi^{2}\nabla_{ref} a}_{L^{3p/2}(M\times [0,t_1])}^{3/2}\leq C\varepsilon \norm{\nabla_{ref}^2 (\varphi^3 a)}_{L^p(M\times [0,t_1])} +C.
\end{equation*}
The proof is the same, except that we use another interpolation inequality
\begin{equation*}
  \norm{\varphi^{2}\nabla_{ref} a}_{L^{3p/2}(B_i)}\leq C\norm{\varphi^3 \nabla^2_{ref} a}^{2/3}_{L^p(B_i)}\norm{a}^{1/3}_{L^4(B_i)}+C\norm{a}_{L^4(B_i)}.
\end{equation*}
By choosing $\varepsilon$ small, we obtain an $W^{2,1}_p$ bound on $a$ for $p=4$, which allows us to apply the estimates for linear parabolic system for higher order estimates. In fact, the parabolic Sobolev embedding theorem in \cite{LSU} implies that $\varphi^2 \partial_t \varphi a$ is in $L^p(M\times [0,t_1])$ for any $p>1$. We then repeat the above argument and use the parabolic Sobolev embedding again to see that $\nabla_{ref} a$ is H\"older continuous. The higher order estimates now follow from Schauder estimates and (\ref{a1}).

\end{proof}

\subsection{Convergence for $t_i\to \infty$}

We now complete the proof of Theorem \ref{thm:one} by considering $t_i\to \infty$. We first claim that we have some gauge transformations $\sigma_i$ such that the $\sigma_i^*(A(t_i))$ are uniformly bounded in any $C^k$ norm. To see this,
let $t_0$ be as in Theorem \ref{thm:globalexist} and set $s_i=t_i-t_0/2$. Consider the solution  $\tilde{A}(t)$ to the modified flow (\ref{eqn:modifiedflow})  with initial value $\tilde{A}(s_i)=A(s_i)$. The proof in Step 2 of Theorem \ref{thm:globalexist} in fact established a $C^k$ estimate for $\tilde{A}(t_i)$, which is gauge equivalent to $A(t_i)$ by the discussion in Subsection \ref{sub:deturk}. Therefore, there is a subsequence which converges smoothly up to gauge transformations. By similar argument above, we have uniform a bound on $\nabla^k F(x,t)$ for any $k$. Due to (\ref{eqn:alphaflow}), we have a uniform bound for $\frac{\partial^k A}{\partial t^k}$ as well. Hence, there is $C>0$ independent of $t$ such that
\begin{equation*}
    \pfrac{}{s} \int_M (1+\abs{F}^2)^{\alpha-1} \abs{\pfrac{A}{s}}^2 dv\leq C.
\end{equation*}
Lemma \ref{lem:energyinequality} then implies that
\begin{equation*}
    \lim_{t\to\infty} \int_M \abs{\pfrac{A}{t}}^2dv =0.
\end{equation*}
Hence, the limit obtained above is a Yang-Mills $\alpha$-connection. This completes the proof of Theorem \ref{thm:one}.

\subsection{Stability of the modified flow}

The results in this subsection are prepared for later applications. Since we shall use the Yang-Mills $\alpha$-flow as a deformation in the space of connections,  we need to show that this flow depends at least continuously on its initial value in some chosen topology.

\begin{thm}
    \label{thm:stable}
    If $D_i=D_{ref}+A_i (i=1,2)$ are two initial connections satisfying
    \begin{equation*}
        \norm{A_i}_{C^{k,\beta}(M)}\leq K,
    \end{equation*}
    then by Theorem \ref{thm:globalexist}, there exists $t_0>0$, which now depends on $K$ and the solution $A_i(t)$ to the modified flow (\ref{eqn:modifiedflow}), which is defined on $[0,t_0]$ and satisfies $A_i(0)=A_i$ and
    \begin{equation*}
        \norm{A_i}_{C^{k,\beta}(M\times [0,t_0])}\leq C(K),
    \end{equation*}
    Moreover, for any $\varepsilon>0$, there exists $\delta(K)>0$ such that if
    \begin{equation*}
        \norm{A_1-A_2}_{C^{k,\beta}(M)}\leq \delta,
    \end{equation*}
    then
    \begin{equation*}
        \norm{A_1(t)-A_2(t)}_{C^{k,\beta}(M)}\leq \varepsilon,
    \end{equation*}
    for $t\in [0,t_0]$.
\end{thm}

\begin{proof}
    The proof of the first part is essentially contained in the proof of Theorem \ref{thm:globalexist}. At that time, we didn't have good control over the initial value, hence a cut-off function in time was used to produce higher order estimates on $M\times [t_0/2,t_0]$. For our purposes here, it suffices to remove the cut-off function $\varphi$ in Step 2 of the proof there.

    The proof of the second part follows from theory of linear partial differential equations and is perhaps well known.
    Both $A_1$ and $A_2$ satisfy the modified Yang-Mills flow, which for our purposes here is written as
        \begin{equation*}
            \pfrac{A_i}{t}=\triangle A_i + (\alpha-1)P(A_i,\nabla A_i)\# \nabla^2 A_i +Q(A_i,\nabla A_i).
        \end{equation*}
    The exact form of $P$ and $Q$ is not important for us. It suffices to know that $P$ and $Q$ are smooth functions of $A_i$ and $\nabla A_i$.
    Subtracting the two equations, we have
        \begin{eqnarray*}
            \pfrac{A_1-A_2}{t}&=& \triangle (A_1-A_2) +(\alpha-1)P(A_1,\nabla A_1)\# \nabla^2 (A_1-A_2) \\
            && + (P(A_1,\nabla A_1)-P(A_2,\nabla A_2))\# \nabla^2 A_2 + Q(A_1,\nabla A_1)-Q(A_2,\nabla A_2).
        \end{eqnarray*}
    There are smooth functions $R$ and $S$ of $A_i$ and $\nabla A_i$ such that
        \begin{eqnarray*}
            \pfrac{A_1-A_2}{t}&=& \triangle (A_1-A_2) +(\alpha-1)P(A_1,\nabla A_1)\# \nabla^2 (A_1-A_2) \\
            && + R(A_i,\nabla A_i,\nabla^2 A_2) (A_1-A_2) + S(A_i,\nabla A_i,\nabla^2 A_2) (\nabla A_1-\nabla A_2).
        \end{eqnarray*}
    If we take the above as a linear parabolic system of $A_1-A_2$, then (1) the system is strictly parabolic in the sense of Petrovskii (note that $P$ is always bounded and hence the principle part is a small perturbation of the Laplacian) and (2) the coefficients are bounded in the $C^{k-2,\alpha}$ norm.

    For the  strictly parabolic linear systems  in the sense of Petrovskii, Eidel'man \cite {Ed} constructed the heat kernel explicitly. Moreover, the solution to the linear system is expressed as the convolution
    \begin{equation*}
        (A_1-A_2)(x,t)=\int_M (A_1-A_2)(y,0) \mathcal Z(x,t;y,0) dv.
    \end{equation*}
    Therefore
        \begin{equation*}
            \norm{A_1-A_2}_{C^0(M\times [0,t_0])}\leq C(K) \norm{A_1(\cdot,0)-A_2(\cdot,0)}_{C^0(M)}.
        \end{equation*}
    We can now apply the Schauder estimate to see
        \begin{eqnarray*}
            && \norm{A_1(\cdot,t)-A_2(\cdot,t)}_{C^{k,\beta}(M)} \\
            &\leq& \norm{A_1-A_2}_{C^{k,\beta}(M\times [0,t_0])} \\
            &\leq& C(K)\norm{A_1(\cdot,0)-A_2(\cdot,0)}_{C^{k,\beta}(M)}.
        \end{eqnarray*}
This proves our claim.
\end{proof}

\section{Convergence of $\alpha$-flow solutions}
\label{sec:convergence}
In this section, we study the convergence of the $\alpha$-flow solutions as $\alpha$ goes to $1$. We follow the same idea as in \cite{HY}. The key ingredients in the proof are a Bochner formula and a monotonicity formula, which are well known techniques but should still be computed for our new equation.

We start with the Bochner formula.
\subsection{Bochner formula and the uniform bound of $F$.}

Let $A(t)$ be a solution of the Yang-Mills alpha flow; i.e.
\begin{equation}\label{eqn:flowa}
  \pfrac{A}{t}=-D^*F+2(\alpha-1)\frac{*( \left <\nabla F,F\right >\wedge *
  F)}{1+\abs{F}^2},
\end{equation}
where $D=D_{ref}+A$. We recall that the curvature $F$ of $D$
satisfies
\begin{equation}\label{eqn:flowF}
    \frac{\partial F}{\partial t}
    =-DD^*F+2(\alpha-1)D\frac{*( \left <\nabla F,F \right >\wedge *F)}{1+\abs{F}^2}.
\end{equation}

For each point $p\in M$, let $e^i$ be a normal frame of $TM$ and
 $\omega^i$ the corresponding orthonormal basis of the cotangent bundle
$T^*M$. Then at $p\in M$,
\begin{equation*}
    F= \sum_{i<j}F_{ij} \omega^i \wedge \omega^j .
\end{equation*}
At $p\in M$,  we can assume that $\nabla e^i=0$ and $\nabla
\omega^i=0$.

In order to derive a Bochner type formula, we need

\begin{lem}  \label{Lemma 5} Let
\[\varphi :=\left <\nabla F,F\right >=\varphi_k \omega^k.\]
Then  at $p\in M$, we have
\[* (\varphi\wedge * F)=\sum_{i=1}^4 \sum_{j=1}^4  \varphi_j F_{ij} \omega^i.\]
\end{lem}
\begin{proof} At $p\in M$, we have
\begin{eqnarray*}
    F&=& F_{12} \omega^1\wedge \omega^2 +F_{13} \omega^1\wedge \omega^3 +F_{14} \omega^1\wedge \omega^4\\
    &+& F_{23} \omega^2\wedge \omega^3 +F_{24} \omega^2\wedge \omega^4 +F_{34} \omega^3\wedge \omega^4.
\end{eqnarray*}
Applying the Hodge star  operator $*$, we have
\begin{eqnarray*}
    *F&=& F_{12} \omega^3\wedge \omega^4 -F_{13} \omega^2\wedge \omega^4 +F_{14} \omega^2\wedge \omega^3\\
    &+&F_{23} \omega^1\wedge \omega^4 -F_{24} \omega^1\wedge \omega^3 +F_{34} \omega^1\wedge \omega^2.
\end{eqnarray*}
Hence
\begin{eqnarray*}
    \varphi\wedge *F&=&+\varphi_1 F_{12} \omega^1\wedge \omega^3\wedge \omega^4 -
    \varphi_1 F_{13} \omega^1\wedge \omega^2\wedge \omega^4  +\varphi_1 F_{14} \omega^1\wedge \omega^2\wedge \omega^3 \\
    &&  +\varphi_2 F_{12} \omega^2\wedge \omega^3\wedge \omega^4 -\varphi_2 F_{23} \omega^1\wedge \omega^2\wedge \omega^4
    +\varphi_2 F_{24} \omega^1\wedge \omega^2\wedge \omega^3 \\
    && + \varphi_3 F_{13} \omega^2\wedge \omega^3\wedge \omega^4 -\varphi_3 F_{23} \omega^1\wedge \omega^3\wedge \omega^4
     + \varphi_3 F_{34} \omega^1\wedge \omega^2\wedge \omega^3 \\
    && +\varphi_4 F_{14} \omega^2\wedge \omega^3\wedge \omega^4 -\varphi_4 F_{24} \omega^1\wedge \omega^3\wedge \omega^4
    +\varphi_4 F_{34} \omega^1\wedge \omega^2\wedge \omega^4.
\end{eqnarray*}
Applying the Hodge star operator again, we have
\begin{eqnarray*}
    *(\varphi\wedge *F)&=&  (\varphi_2 F_{12}+ \varphi_3 F_{13} +\varphi_4 F_{14}) \omega^1 \\
    &+& (-\varphi_1 F_{12}+\varphi_3 F_{23} +\varphi_4 F_{24}) \omega^2 \\
    &+& (-\varphi_1 F_{13}-\varphi_2 F_{23}+\varphi_4 F_{34}) \omega^3 \\
    &+& (-\varphi_1F_{14}-\varphi_2 F_{24} -\varphi_3 F_{34}) \omega^4 \\
    &=& \sum_{i=1}^4 \sum_{j=1}^4  \varphi_j F_{ij} \omega^i.
\end{eqnarray*}
This proves our claim.
\end{proof}

\begin{lem} \label{lem:firstbochner} (Bochner type  formula 1)
 When $\alpha-1$ is sufficiently small, there is a constant $C$ such that
  \begin{eqnarray} \label{eqn:firstbochner}
  && \frac{\partial}{\partial t}\abs{F}^2 - \nabla_{e_i} \left( (\delta_{ij} +2(\alpha-1)\frac{
   \left <F_{lj},F_{li}\right >}{1+\abs{F}^2}) \nabla_{e_j} \abs{F}^2\right)+|\nabla F|^2 \\
   &&\leq C\abs{F}^2 (1+\abs{F}).\nonumber
  \end{eqnarray}
\end{lem}
\begin{proof}
    Recall that we use a local normal orthonomal frame $\set{e_i}$ and its dual $\set{\omega_i}$ at $p$. Noticing the fact that $\nabla_{e^j} e^j=0$ at $p\in M$, we have
\begin{equation*}
    \nabla^*\nabla |F|^2= -\sum_j \nabla^2_{e^j,e^j} |F|^2=-\sum_j \nabla_{e^j} \nabla_{e^j}
    |F|^2
\end{equation*}
and
\begin{eqnarray*}
   \sum_{i} \nabla^2_{e^i;e^i} \left <F,F \right >= 2\sum_{i}\left <\nabla_{e_i}F,\nabla_{e_i} F \right >
   + 2\sum_{i}\left < F,\nabla_{e_i; e_i} F \right >.
\end{eqnarray*}
The well-known Weizenb\"ock
formula is
\begin{eqnarray*}
  \triangle F
  = \nabla^* \nabla F + F\circ (Ric\wedge g+2R)+ F\# F.
\end{eqnarray*}
Here $Ric$ is the Ricci curvature of $M$ and $R$ is the curvature operator, $(Ric\wedge g+2R)$ is a linear mapping from 2 forms to 2 forms. We refer to Theorem (3.10) of \cite{BL} for the exact statement and the proof. Since we are not interested in the exact form of the last term and it is quadratic in $F$, we denote it by $F\# F$.

Using Bianchi's identity $DF=0$,
we have
\begin{eqnarray*}
    -\left <DD^* F,F\right >=\left <\nabla_{e^i} \nabla_{e^i} F + F\#F - F\circ (Ric\wedge g+2R), F\right>.
\end{eqnarray*}

  For
simplicity, we set
\begin{equation*}
    b_{ij}=2(\alpha-1)\frac{\left <F_{lj},  F_{li}\right >}{1+\abs{F}^2}.
\end{equation*}
Then we have
\begin{eqnarray}\label{eqn:c1}
  &&\pfrac{}{t} \abs{F}^2 - \nabla_{e^i} \left( (\delta_{ij}+ b_{ij})
    \nabla_{e^j} \abs{F}^2 \right)\\
    &=& 2\left <F,\pfrac{}{t}F\right > - 2\nabla_{e^i} \left <\nabla_{e^i} F, F \right >
    -\nabla_{e^i} \left( b_{ij} \nabla_{e^j}\abs{F}^2 \right) \nonumber\\ \nonumber
    &=& 2\left <F,\pfrac{F}{t}+DD^* F\right > + \left <F, F \#F - F\circ (Ric\wedge g+ 2R)\right > \\
    &&
    -2 \abs{\nabla F}^2-\nabla_{e^i} \left( b_{ij} \nabla_{e^j} \abs{F}^2 \right).
    \nonumber
\end{eqnarray}

By Lemma \ref{Lemma 5}, we have
\begin{equation*}
  * (\varphi\wedge *F)=\sum \varphi_i F_{ij} \omega^j .
\end{equation*}
Let $f(|a|)$  denote a function , whose absolute value is smaller
than a constant multiple of $|a|$; i.e. $|f(a)|\leq C|a|$ for a
constant $C>0$. Then at $p$, we have
\begin{eqnarray*}
    D\frac{*(\varphi \wedge *F)}{1+\abs{F}^2} &=& \frac{D( *(\varphi\wedge *F))}{1+\abs{F}^2}
     + d (1+|F|^2)^{-1}\wedge *(\varphi\wedge *F) \\
    &{=}& \frac{(\varphi_i F_{ij})_{;k} \omega^k\wedge \omega^j}{ 1+\abs{F}^2}
    +f(\abs{\nabla F}^2 \frac{1}{\abs{F}}) \\
    &{=}&
     \frac{\varphi_{i;k} F_{ij}\omega^k\wedge \omega^j}{ 1+\abs{F}^2} +f(\abs{\nabla F}^2
     \frac{1}{\abs{F}})\\
\end{eqnarray*}
which implies
\begin{eqnarray*}
\left <D\frac{*(\varphi\wedge *F)}{1+\abs{F}^2},\,F\right >
     &=&
    \frac{ \varphi_{i;k}\left <F_{ij}, F_{kj}\right >}{1+\abs{F}^2} +f(\abs{\nabla F}^2).
\end{eqnarray*}

On the other hand, we have at $p$
\begin{eqnarray}\label{eqn:c2}
     &&\nabla_{e^i} \left( b_{ij} \nabla_{e^j} \abs{F}^2 \right)\\ \nonumber
     &=&
     \nabla_i \left( 4(\alpha-1)\frac{\left <F_{lj},\,F_{li}\right >}{1+\abs{F}^2} \varphi_j\right)
     \\ \nonumber
     &=&
     4(\alpha-1)\frac{\left <F_{lj},\, F_{li}\right >}{1+\abs{F}^2} \varphi_{j;i} +(\alpha -1)
     \frac {F\#F\#\nabla F\#\nabla F}{1+|F|^2}\\
     &&+(\alpha -1)
     \frac {F\#F\#\left <F,\nabla F\right >^2}{(1+|F|^2)^2}\nonumber\\
     &\geq &  4(\alpha-1)\left <D\frac{*(\varphi\wedge *F)}{1+\abs{F}^2},\,F\right >- C (\alpha -1)\abs{\nabla F}^2.\nonumber
  \end{eqnarray}
  Since $p$ is an arbitrary point of $M$, we may combine
  (\ref{eqn:flowF}), (\ref{eqn:c1}) and (\ref{eqn:c2}) to get (when $\alpha-1$ small),
  \begin{eqnarray} \label{eqn:betterbochner}
      && \frac{\partial}{\partial t} \abs{F}^2-\nabla_{e_i} \left( (\delta_{ij}+b_{ij})\nabla_{e_j}\abs{F}^2 \right)+\abs{\nabla F}^2 \\ \nonumber
      &\leq& C \abs{F}^3 -\left< F, F\circ(Ric\wedge g+2R)\right>.
  \end{eqnarray}
  Since the manifold is compact and the curvatures are bounded, the lemma follows trivially from (\ref{eqn:betterbochner}). We shall use this shaper estimate later to prove a gap theorem for Yang-Mills $\alpha$-connections on $S^4$.
\end{proof}

As a consequence of Lemma \ref{lem:firstbochner} , we have
\begin{lem}(Bochner type  formula 2) \label{lem:secondbochner} For each $\alpha >1$, let $A$ be the smooth solution of
the Yang-Mills $\alpha$-flow and $F:=F_A$ the curvature of $A$.
Then
  for $\alpha-1$ sufficiently small, we have
  \begin{eqnarray}\label{eqn:secondbochner}
    &&
   \frac{\partial}{\partial t} (1+|F|^2)^{\alpha} - \nabla_{e_i} \left( (\delta_{ij} +2(\alpha-1)\frac{
   \left <F_{lj},F_{li}\right >}{1+\abs{F}^2}) \nabla_{e_j}  (1+|F|^2)^{\alpha}\right)\\ \nonumber
   && \leq
C (1+|F|^2)^{\alpha} (1+\abs{F})
\end{eqnarray}
  for a constant $C>0$.
\end{lem}
\begin{proof} In fact, one sees
\[\pfrac{}{t}(1+|F|^2)^{\alpha}= \alpha (1+|F|^2)^{\alpha -1}\pfrac{\abs{F}^2}{t}\]
and
\[ \nabla_{e_j}(1+|F|^2)^{\alpha} = \alpha (1+|F|^2)^{\alpha -1}\nabla_{e_j} |F|^2 .\]
For simplicity, we set
\[ a_{ij}= \delta_{ij} +2(\alpha-1)\frac{
   \left <{F}_{lj},{F}_{li}\right >}{1+\abs{F}^2} .\]
   Then we have
\begin{eqnarray*}
    && \nabla_{e_i} \left ( a_{ij} \nabla_{e_j} (1+|F|^2)^{\alpha} \right )\\
 &=&\alpha \nabla_{e_i} (  a_{ij} (1+|F|^2)^{\alpha -1}\nabla_{e_j} |F|^2   ) \\
    &=& \alpha (1+|F_{\alpha}|^2)^{\alpha -1} \nabla_{e_i} (a_{ij} \nabla_{e_j} F|^2  )\\
    &&+\alpha (\alpha -1) (1+|F|^2)^{\alpha -2}  a_{ij} \nabla_{e_i} |F|^2\nabla_{e_j}|F|^2.
\end{eqnarray*}
By Lemma  \ref{lem:firstbochner}, we obtain
\begin{eqnarray*}
    && \frac{\partial}{\partial t} (1+|F|^2)^{\alpha} - \nabla_{e_i} \left ( a_{ij} \nabla_{e_j}  (1+|F|^2)^{\alpha} \right )\\
    &&= \alpha (1+|F|^2)^{\alpha -1} \left [\frac{\partial}{\partial t} |F|^2  - \nabla_{e_i} (a_{ij} \nabla_{e_j}  |F|^2  )\right ] \\
    &&-\alpha (\alpha -1) (1+|F|^2)^{\alpha -2}  a_{ij} \nabla_{e_i}
    |F|^2\nabla_{e_j}|F|^2.\\
    &&\leq C (1+|F|^2)^{\alpha -1} |F|^2
    (1+|F|).
\end{eqnarray*}
This proves our claim.
\end{proof}

\subsection{Monotonicity formula}

The global parabolic monotonicity formula for  harmonic maps was
first established by Struwe in \cite {St1},  and for the
Yang-Mills flow in \cite {CS} and \cite {Ha}. Next, we will
derive a local parabolic type of monotonicity for the Yang-Mills
$\alpha$-flow as similar to one in \cite {HT1}.

Let $i(M)$ be the injectivity radius of $M$. Consider some fixed $x_0\in M$ and let $\phi$ be a cut-off function supported in $B_{i(M)}(x_0)$ with $\phi\equiv 1$ on $B_{i(M)/2}(x_0)$.
 For
$z_0=(x_0,t_0)\in M\times \R_+$, we write
$$T_R(z_0)=\left \{ z=(x,t) : t_0-4R^2<t<t_0-R^2, x\in M\right \}\,.$$
When there is no ambiguity for $z_0$, we write $T_R$ only.

If we take the normal coordinates $\{x^i\}$ in $B_{i(M)}(x_0)$, the Euclidean backward heat kernel to the (backward)
heat equation with singularity at $z_0$ is
$$G_{z_0}(z)=\frac 1{(4\pi (t_0-t))^{2}}\text{exp} \left (-\frac
{|x|^2}{4(t_0-t)}\right )\,, \quad t<t_0.$$
As before, we write $G(x,t)$ when $z_0$ is obvious.

Assume that  $A$ is a solution of the $\alpha$-flow (1.4) in
$M\times \R_+$.  For any $z_0=(x_0, t_0)\in M\times [0,T]$, we set
\begin{equation}\label{eqn:Phi}
\Phi_{\alpha} (R, A;  z_0)= R^{4\alpha -2} \int_{T_R(z_0)}  (1+|F(z)|^2)^{\alpha} \,\phi^2(x-x_0)\,G_{z_0}(z)\,dv\,dt.
\end{equation}

\begin{lem} \label{Lemma 3.2}  (Local Monotonicity)  Let $A$ be a
regular solution of the $\alpha$-flow (1.4). Then, for
$z_0=(x_0,t_0)\in M\times (0,\infty )$ and for any two numbers
$R_1$, $R_2$ with $0<R_1\leq R_2\leq i(M)$, we have
\begin{eqnarray*}
    \Phi_{\alpha} (R_1,A;  z_0) \leq  C\,  \mbox {exp}(C(R_2-R_1))
 \Phi_{\alpha} (R_2,A;  z_0) + C(R^2_2 -R^2_1)\mbox
{YM}_{\alpha}(A_0).
\end{eqnarray*}
\end{lem}
\begin{proof} Although the main idea of the proof is similar to one for the
Yang-Mills flow in \cite {HT1},  the proof becomes much more involved,
so we have to give more details here.

Since the computation is local, we choose normal coordinates $\{x^i\}$ around $x_0$ and assume without loss of generality that $t_0=0$.

In (\ref{eqn:Phi}), we set $x=R\tilde{x}$ and $t=R^2\tilde{t}$ to obtain
$$ \Phi_{\alpha} (R, A;  z_0)=\int_{T_1} R^{4\alpha }(1+|F|^2(x,t))^{\alpha} \, \phi^2 (R\tilde
x)\,G(\tilde z)\sqrt { g (R\tilde x)}\,d\tilde z\,,
  $$
where $d\tilde z=d\tilde x\,d\tilde t$.

Then we compute
\begin{eqnarray*}
    && \frac d{dR} \Phi_{\alpha} (R, A;  z_0) =\int_{T_1}\frac d{dR}\left [   R^{4\alpha } [1+| F|^2 (R\tilde
x, R^2\tilde t)]^{\alpha}  \,\phi^2 (R\tilde x)\, \sqrt {g
(R\tilde x)} \,\right ]  G(\tilde z) \,d\tilde z \\
    && =4\alpha R^{4\alpha -1}\int_{T_1}  [1+| F|^2 (R\tilde
x, R^2\tilde t)]^{\alpha}  \,\phi^2 (R\tilde x) \sqrt {g
(R\tilde x)}\,G(\tilde z)\,d\tilde z\\
    &&  + \alpha R^{4\alpha}\int_{T_1} [1+| F|^2 (R\tilde
x, R^2\tilde t)]^{\alpha -1} \tilde x^k\frac{\partial }{\partial
x^k} |F|^2 (R\tilde x, R^2\tilde t) \,\phi^2(R\tilde x)\, \sqrt
{g (R\tilde x)} G(\tilde
z)\,d\tilde z\\
    && + \alpha R^{4\alpha}\int_{T_1}[1+| F|^2 (R\tilde
x, R^2\tilde t)]^{\alpha -1}2R\tilde t\frac{\partial }{\partial
t}|F|^2 (R\tilde x, R^2\tilde t)\, \phi^2(R\tilde x)\, \sqrt {g
(R\tilde x)}\,G(\tilde
z)\,d\tilde z\\
    && +\int_{T_1} R^{4\alpha } [1+| F|^2 (R\tilde
x, R^2\tilde t)]^{\alpha}  \, \tilde x^k\frac{\partial }{\partial
x^k} \,(\phi^2  \sqrt {g}) (R\tilde x)\, G(\tilde z)
\,d\tilde z \\
    && := I_1+ I_2+I_3+I_4\,.
\end{eqnarray*}

In order to estimate  $I_1$ and $I_2$, we note that in local
coordinates, we have
$$F = \frac 1 2  F_{ij}dx^i\wedge dx^j.
$$
Let $\nabla_{A, x^k} F=\frac 12 \nabla_{A, x^k} F_{ij}
dx^i\wedge dx^j $ be the gauge-covariant derivative of $F$ with
respect to $\frac {\partial}{\partial x^k}$ satisfying $\nabla_{A,
x^k} F_{ij}=\frac{\partial F_{ij}}{\partial x^k}+[A_k,
F_{ij}]-\sum_s\Gamma^s_{ik} F_{sj}-\sum_s\Gamma^s_{jk} F_{is}$.
Since $A$ is compatible with the Riemannian structure, we have
$$\frac{\partial }{\partial x^k}|F|^2 =\frac 1 2
\left <\nabla_{A,x^k}F_{ij}dx^i\wedge dx^j, F_{lm}dx^l \wedge
dx^m\right >\,.$$ In local coordinates, the Bianchi identity
$DF=0$   is equivalent to
$$\nabla_{A, x^k} F_{ij}=\nabla_{A,
x^i}F_{kj}-\nabla_{A, x^j} F_{ki}.$$ Using the Bianchi identity,
we have
\begin{eqnarray*}
    && x^k\frac{\partial }{\partial x^k}|F|^2 =\frac 1 2
x^k\left <(\nabla_{A,x^i}F_{kj}-\nabla_{A,x^j}F_{ki})dx^i\wedge
dx^j,\,F_{lm}dx^l\wedge dx^m\right > \\
    &&=  \left <\nabla_{A,x^i}(x^kF_{kj}) dx^i\wedge dx^j,\,
F_{lm}dx^{l}\wedge
dx^{m}\right> -4|F|^2\\
&&\quad -\left <x^kF_{sj}\Gamma^s_{ki}dx^i\wedge dx^j,\,
F_{lm}dx^{l}\wedge dx^{m} \right >,
\end{eqnarray*}
where $\nabla_{A, x^i} (x^kF_{kj}):=\frac{\partial }{\partial
x^i}(x^kF_{kj})+[A_i, x^k F_{kj}]-x^kF_{ks}\Gamma_{ji}^s$ is the
gauge-covariant derivative of $x^kF_{kj}$ with respect to $\frac
{\partial}{\partial x^i}$. Changing back to $(x,t)$, we have
\begin{eqnarray*}
    && I_1+I_2=\alpha R^{4\alpha -3}\int_{T_R}(1+|F|^2)^{\alpha -1}[4(|F|^2+1)+  x^k\frac {\partial |F|^2}{\partial
x^k}  ]\,\phi^2 \,G\,\sqrt {g}\,dz\\
    &&=\alpha R^{4\alpha -3}\int_{T_R}(1+|F|^2)^{\alpha -1}[4 +\left <\nabla_{A,x^i}(x^kF_{kj}) dx^i\wedge dx^j,\,
F_{lm}dx^{l}\wedge dx^{m}\right> ]\,\phi^2 \,G\,\sqrt {g}\,dz
     \\
     &&\quad -\alpha R^{4\alpha -3}\int_{T_R}(1+|F|^2)^{\alpha -1}\left
<x^kF_{sj}\Gamma^s_{ki}dx^i\wedge dx^j,\, F_{lm}dx^{l}\wedge
dx^{m} \right >\,\phi^2 \,G\,\sqrt {g}\,dz.
\end{eqnarray*}
Note that
$$D^* [(1+|F|^2)^{\alpha -1}F]=-g^{il} \nabla_{A, x^i}[ (1+|F|^2)^{\alpha -1}F_{lm}]dx^{m}\,.$$
Then  using Stokes' formula, we have
\begin{eqnarray*}
    && \int_{T_{R}}(1+|F|^2)^{\alpha -1}\left <\nabla_{x^i}(x^kF_{kj}) dx^i\wedge dx^j,\,
F_{lm}dx^{l}\wedge dx^{m} \right >\,\phi^2\, G\,\sqrt g\,dz \\
    && =2\int_{T_{R}}(1+|F|^2)^{\alpha -1}\left <\nabla_{x^i}(x^kF_{kj})  dx^j,\,
g^{il}F_{lm}dx^{m}\right >\,\phi^2\, G\,\sqrt g\,dz\\
    && =2\int_{T_{R}}\left < x^kF_{kj}dx^j,\,D^* [(1+|F|^2)^{\alpha -1}F] \right >\,\phi^2 \,G\,\sqrt g\,dz \\
    && \quad -2\int_{T_{R}}(1+|F|^2)^{\alpha -1}\left <x^kF_{kj}dx^j,\,
g^{il}F_{lm}dx^{m}\right >\,\phi^2\,\frac
{\partial G}{\partial x^i}\sqrt g\,dz \\
    && -4\int_{T_{R}}(1+|F|^2)^{\alpha -1}\left <x^kF_{kj}dx^j,\,
g^{il}F_{lm}dx^{m}\right >\,\phi\,\frac {\partial \phi
}{\partial x^i}\,G\,\sqrt g\,dz\,.
\end{eqnarray*}

Using the fact that
$$|g_{ij}(x)-\delta_{ij}|\leq C|x|^2, \quad \left |\frac {\partial
g_{ij}}{\partial x^k}\right |\leq C |x|,\quad \frac {\partial
G}{\partial x^i}=\frac {x^i}{2t}G,$$ we have
\begin{eqnarray*}
     I_1+I_2&\geq & 2\alpha R^{4\alpha -3}\int_{T_{R}} \<x^kF_{kj}dx^j,
    D^*((1+|F|^2)^{\alpha -1}
F)\>\,\phi^2\,
G\,\sqrt g\,dz \\
&&  +\alpha R^{4\alpha -3}\int_{T_{R}}(1+|F|^2)^{\alpha -1} |x^ig^{ik} F_{kj}dx^j|^2\frac{1}{\abs{t}}  \,G\,\phi^2\,\sqrt g\,dz \\
    &&   -C\alpha R^{4\alpha -3}\int_{T_{R_1}} (1+|F|^2)^{\alpha}(|x|^2 \phi^2 +|x| |\nabla \phi
    | +\frac {|x|^4}{\abs{t}} \phi^2) \,\,G\,\sqrt g\,dz\,.
\end{eqnarray*}
To estimate $I_3$,we note that the $\alpha$-flow (1.4) is
equivalent to
\[ (1+\abs{F}^2)^{\alpha-1} \frac {\partial A}{\partial t}=-D^*\left(  (1+\abs{F}^2)^{\alpha-1} F \right).\]
Then using Stokes' formula, we have
\begin{eqnarray*}
    && I_3=2\alpha R^{4\alpha -3}\int_{T_{R}}(1+|F|^2)^{\alpha -1} \,t\frac {\partial}{\partial t}
|F|^2\,\phi^2\,G\,\sqrt g\, dz \\
    && =4\alpha R^{4\alpha -3} \int_{T_R}t\<(1+|F|^2)^{\alpha -1} F,
\,D(\frac {\partial A}{\partial t})\> \,\phi^2\,G\,\sqrt g dz \\
    &&=4\alpha R^{4\alpha -3}\int_{T_{R}} t \<D^*\left [(1+|F|^2)^{\alpha -1} F\right ],
\,\frac {\partial A}{\partial t}\> \,\phi^2\,G\,\sqrt g\,
dz \\
    &&  - 4\alpha R^{4\alpha -3}\int_{T_{R}}t (1+|F|^2)^{\alpha -1}\<\frac {\partial A}{\partial t},\,
g^{il}F_{lm}dx^{m}\> \(\frac {\partial  G}{\partial
x^i}\phi^2+2\phi \frac {\partial \phi }{\partial x^i}G\)\,\sqrt
g\, dz
\\
&&=4\alpha R^{4\alpha -3}\int_{T_{R}} \abs{t} (1+|F|^2)^{\alpha -1}
|\frac {\partial A}{\partial t}|^2 \,\phi^2\,G\,\sqrt g\,
dz \\
    && - 2\alpha R^{4\alpha -3}\int_{T_{R}}  (1+|F|^2)^{\alpha -1} \<\frac {\partial A}{\partial t},\,
 x^ig^{il}F_{lm}dx^{m}\>    \phi^2 G\,\sqrt g\, dz
\\
    && - 4\alpha R^{4\alpha -3}\int_{T_{R}}t (1+|F|^2)^{\alpha -1} \<\frac {\partial A}{\partial t},\,
g^{il}F_{lm}dx^{m}\>  2\phi \frac {\partial \phi }{\partial
x^i} G\,\sqrt g\, dz.
\end{eqnarray*}
Using above estimates and also Young's inequality, we obtain
\begin{eqnarray*}
    &&\frac d{dR}\Phi (R;A)=I_1+I_2+I_3+I_4 \\
    && \geq \frac 12  \alpha R^{4\alpha -3}\int_{T_{R}} |t| (1+|F|^2)^{\alpha -1}  \left | 2 \frac {\partial A}{\partial
t}-\frac {x^i} {|t|} g^{il}F_{lm}dx^{m} \right
|^2  \,\phi^2\, G\,\sqrt g\,dz \\
    && -C R^{4\alpha -3}\int_{T_{R}}  (1+|F|^2)^{\alpha }
(|x|^2\phi^2+ |x| |\nabla  \phi |+\frac {|x|^4}{t} \phi^2 + |t|
|\nabla \phi |^2) \,G\,\sqrt g\,d z,
\end{eqnarray*}
where $C$ is a constant depending on the geometry of $M$. We know
 that
$$R^{-1} |x|^2G\leq C(1+G), \quad R^{-1} |t|^{-1}|x|^4G\leq
C(1+G)\quad \text {on }T_{R}\,.$$ Moreover, since $\nabla \phi =0$
for $\abs{x}<i(M)/2$,
 we see that
 $$ (|x| |\nabla \phi |+\abs{t} |\nabla \phi|^2)  G\leq  C \quad \mbox{on } T_R.$$
 Combining these estimates with Lemma \ref{lem:energyinequality}, we obtain
\begin{eqnarray*}
   \frac d{dR}\Phi (R; A)\geq -C \Phi (R; u, A)-CR\mbox
   {YM}_{\alpha}(A_0)\,.
\end{eqnarray*}
The claim for $\Phi$ follows from integrating the above inequality
in $R$.
\end{proof}

\subsection{The $\varepsilon-$regularity and convergence}
\begin{lem} \label{lem:eregularity}  There exists a positive constant
    $\varepsilon_0<i(M)$ such that for a solution $A$ to (\ref{eqn:alphaflow}), if for some $R$ with $0<R<\min
\{\varepsilon_0, \frac {t_0^{1/2}}2\}$  the inequality
$$R^{4\alpha -6}\int_{P_R(x_0, t_0)}(1+|F|^2)^{\alpha} \,dv\,dt
    \leq \varepsilon_0$$ holds,
we have
$$\sup_{P_{\frac 14 R}(x_0,t_0)} |F|^2 \leq C R^{-4}\,,$$
where the constant $C$   depends on $M$ and the bound of $\mbox
{YM}_{\alpha} (A_0)$.
\end{lem}
\begin{proof}
 Without loss of generality, assume that
$(x_0,t_0)=(0,0)$. For simplicity, we set $r_1=\frac 1 2 R$. As in
\cite {Sch}, we choose $r_0<r_1$ such that
\[
(r_1-r_0)^{4\alpha }\sup_{{ P}_{r_0}}(1+|F|^2)^{\alpha} =
   \max_{0\leq r\leq r_1}\left [ (r_1-r)^{4\alpha }\sup_{P_r} (1+|F|^2)^{\alpha}\right ],  \] and find $(x_1,t_1)\in P_{r_0}$ such that
$$e_0:=(1+|F|^2)^{\alpha} (x_1,t_1)=\sup_{{P}_{r_0}} (1+|F|^2)^{\alpha} \,.$$
 We claim that
\begin{eqnarray}\label{3.1}
e_0\leq 2^{4\alpha} (r_1-r_0)^{-4\alpha}\,.\end{eqnarray}
 Otherwise, we have
$$\rho_0=e_0^{-\frac 1{4\alpha }}\leq \frac{r_1-r_0}2\,.$$
Rescale
$$B(\tilde x)=\rho_0\,A(x_1+\rho_0\tilde x, t_1+\rho_0^2\tilde t)\,.$$
and
$$
e_{\rho_0}:= (\rho_0^4+| F_B|^2)^{\alpha} =\rho_0^{4\alpha
}\,(1+|F|^2)^{\alpha}.$$ Then we have
\begin{eqnarray*}
    && 1=e_{\rho_0}(0,0)\leq \sup_{\bar P_1}e_{\rho_0}(\tilde x,
    \tilde t)
  = \rho_0^{4\alpha}\sup_{P_{\rho_0} (x_1,t_1)} (1+|F(x,t)|^2)^{\alpha}\\
    && \leq \rho_0^{4\alpha}\(\frac{r_1 - r_0}{2}\)^{-4\alpha}\(\frac{r_1 -
    r_0}{2}\)^{4\alpha}
      \sup_{P_{\frac{r_1+r_0}2}} (1+|F(x,t)|^2)^{\alpha}\\
    &&\leq \rho_0^{4\alpha}\(\frac{r_1 - r_0}{2}\)^{-4\alpha}\(r_1 - r_0\)^{4\alpha}e_0
     = 2^{4\alpha},
\end{eqnarray*}
with $\tilde P_1 :=\{ (\tilde x,\tilde t): \quad (\tilde x,\tilde
t)\in B_1(0)\times [-1,1]\}$.
 This implies that
$$|F_B|^2\leq 16\quad \text{on } \bar P_1.$$
Combining this with Lemma \ref{lem:secondbochner}, we have
\begin{eqnarray*}
    && (\frac {\partial}{\partial \tilde  t}e_{\rho_0} -\tilde \nabla_{e_i} \left( (\delta_{ij} +2(\alpha-1)\frac{
   \left <{F_B}_{lj},{F_B}_{li}\right >}{\rho_0^4+\abs{F_B}^2}) \tilde \nabla_{e_j}
   e_{\rho_0}\right)\\&&
 =\rho_0^{2+4\alpha} \,\left [\frac {\partial}{\partial  t}(1+|F|^2)^{\alpha}
-\nabla_{e_i} \left( (\delta_{ij} +2(\alpha-1)\frac{
   \left <{F}_{lj},{F}_{li}\right >}{1+\abs{F}^2}) \nabla_{e_j}
   (1+|F|^2)^{\alpha}\right)\right ]
\\
    && \leq C e_{\rho_0},
\mbox { in } \tilde P_1\,,
\end{eqnarray*}
where the constant $C$ depends on $i(M)$ and $\sup_{x\in M}
|R_m|$. Then  Moser's parabolic Harnack inequality yields
\begin{eqnarray} \label{3.2}
    && 1=e_{\rho_0}(0,0)\leq C\int_{\tilde P_1}e_{\rho_0}\,d\tilde x\,d\tilde t= C\rho_0^{4\alpha -6}\int_{P_{\rho_0}(x_1,t_1)}  (1+|F|^2)^{\alpha}  \,dv\,dt\,.
\end{eqnarray}

Taking $\sigma =2\rho_0$ and noting that $z_1=(x_1,t_1)\in
P_{r_0}$ and $\sigma +r_0\leq \frac R2$, we apply Lemma \ref{Lemma
3.2} with $R_1=\frac {\sigma}2$, $R_2=\frac 12 R$ to obtain
\begin{eqnarray}\label{3.3}
    &&\rho_0^{4\alpha -6}\int_{P_{\rho_0}(z_1)} (1+|F|^2)^{\alpha}\,dv\,dt \\
    && \leq C\int_{T_{\sigma }(x_1,t_1+2\sigma^2)} \sigma^{4\alpha -2} (1+|F|^2)^{\alpha}\,
G_{(x_1, t_1+2\sigma
^2)}\,\phi^2\,dv\,dt\nonumber\\
    && \leq C\int_{T_{\frac  12 R}(x_1,t_1+2\sigma^2)}R^{4\alpha -2} (1+|F|^2)^{\alpha}\, G_{(x_1,t_1+2\sigma^2)}\,\phi^2\,dv\,dt\nonumber\\
    &&\quad +CR \mbox {YM}_{\alpha} (A_0) +CR^{4\alpha -6}\int_{P_{\frac R 2}(x_1,t_1+2\sigma^2)}(1+|F|^2)^{\alpha}\,dv\,dt\nonumber\\
    &&\leq CR^{4\alpha -6}\int_{P_R}  (1+|F|^2)^{\alpha}\,dv\,dt +CR\,
\mbox {YM}_{\alpha}(A_0)\leq C\varepsilon_0,\nonumber
\end{eqnarray}
where we used the fact that  for $t_1+2\sigma^2 -R^2\leq t\leq
t_1+2\sigma^2-\frac {R^2}4$ and $x\in B_R(x_0)$, there is a
constant $C$ such that
$$G_{x_1,t_1+2\sigma^2} =  \frac 1{(4\pi (t_1+2\sigma^2-t))^{2}}\text{ exp} \left (-\frac
{(x-x_1)^2}{4(t_1+2\sigma^2-t)}\right ) \leq CR^{-4}.
$$
Letting $\varepsilon_0$  be sufficiently small, (\ref{3.3})
contradicts (\ref{3.2}). Therefore, we have proved the claim
(\ref{3.1}), which implies
$$\sup_{P_{R/4}} (1+|F|^2)^{\alpha}
    \leq   (\frac {r_1} 2)^{-4\alpha }(r_1-r_0)^{4\alpha}e_0\leq 2^{4\alpha}  R  ^{-4\alpha}\,.$$
This proves Lemma \ref{lem:eregularity}.
 \end{proof}

 With the curvature bound obtained by Lemma \ref{lem:eregularity}, we may obtain higher order derivative estimates of $F$.

 \begin{lem}
     \label{lem:highorder} Suppose that $A$ is a solution of the flow equation (\ref{eqn:flowa}) on some parabolic ball $P_r(x_0,t_0)$ and that
   \begin{equation*}
     \sup_{P_r(x_0,t_0)} \abs{F}\leq C.
   \end{equation*}
   Then for each $k$, there is a constant $C_k$ such that
   \begin{equation*}
     \sup_{P_{r/2}(x_0,t_0)} \abs{\nabla^k F} \leq C(k).
   \end{equation*}
 \end{lem}

 \begin{proof}
  Assume that $r=1$ and write $P_r$ for $P_r(x_0,t_0)$. Recall that $F$ satisfies
  \begin{equation*}
    \pfrac{F}{t}=-DD^*F+2(\alpha-1) D\frac{*(\langle \nabla F,F\rangle\wedge *F)}{1+\abs{F}^2}.
  \end{equation*}
  By the Bianchi identity and Weizenb\"ock formula, we have
  \begin{equation}\label{eqn:forF}
    \pfrac{F}{t}=\triangle F+2(\alpha-1) D\frac{*(\langle \nabla F,F\rangle\wedge *F)}{1+\abs{F}^2} + F\# F+ \mbox{Rm}\#F,
  \end{equation}
  where $\triangle$ is the covariant Laplacian and $\mbox{Rm}$ is the Riemannian curvature of $M$.
  The proof is by induction. Let $\varphi$ be a cut-off function supported in $B_1$ with $\varphi\equiv 1$ on $B_{3/4}$. Multiplying both sides of (\ref{eqn:forF}) by $\varphi^2 F$ and integrating over $B_1$, we have
  \begin{equation*}
    \frac{1}{2}\frac{d}{dt}\int_{B_1} \varphi^2 \abs{F}^2 dv +\int_{B_1} \varphi^2 \abs{\nabla F}^2 dv \leq C (\alpha-1) \int_{B_1} \varphi^2 \abs{\nabla F}^2 dv + \mathcal L,
    \end{equation*}
  where $\mathcal L$ contains all `lower order terms'.

  In the above equation, it includes $\int_{B_1} \varphi^2 \abs{F}^3 dv$ and $\int_{B_1} \varphi^2 \abs{F}^2dv$, which are bounded, and $\int_{B_1} \abs{\nabla\varphi} \varphi \abs{\nabla F} \abs{F} dv$, which arises in the integration by parts. We shall see that
  \begin{equation}\label{eqn:lot1}
    \mathcal L \leq \eta \int_{B_1} \varphi^2\abs{\nabla F}^2 dv +C.
  \end{equation}
  In fact,
  \begin{equation*}
    \int_{B_1} \abs{\nabla \varphi}\varphi \abs{\nabla F} \abs{F} dv \leq C + \eta\int_{B_1} \varphi^2 \abs{\nabla F}^2 dv.
  \end{equation*}
  By choosing $\alpha-1$ and $\eta$ small, we conclude that
  \begin{equation*}
    \int_{P_{3/4}} \abs{\nabla F}^2 dv dt\leq C.
  \end{equation*}
  We may choose a good time slice on which the space integration of $\abs{\nabla F}^2$ is bounded. Instead of further shrinking the neighborhood, we assume $\int_{P_1} \abs{\nabla F}^2 dv dt\leq C$ and $\int_{B_1} \abs{\nabla F}^2 (\cdot,-1) dv\leq C$, which is the starting point for the next step of induction.

Applying  $\nabla$ on (\ref{eqn:forF}), multiplying by $\varphi^4 \nabla F$ and integrating over $B_1$, we have
  \begin{equation*}
    \frac{1}{2}\frac{d}{dt}\int_{B_1} \varphi^4 \abs{\nabla F}^2 dv +\int_{B_1} \varphi^4 \abs{\nabla^2 F}^2 dv \leq C (\alpha-1) \int_{B_1} \varphi^4 \abs{\nabla^2 F}^2 + \varphi^4 \abs{\nabla F}^4 dv + \mathcal L.
  \end{equation*}
  The lower order terms (still denoted by $\mathcal L$) which arise from switching the order of covariant derivatives, integration by parts and interchanging $\nabla$ and $\pfrac{}{t}$ can be controlled by $\eta\int_{B_1} \varphi^4 \abs{\nabla F}^4+ \varphi^4 \abs{\nabla^2 F}^2 dv +C$ as before. For example,
  \begin{eqnarray*}
    \int_{B_1} \abs{\nabla^2 F} \abs{\nabla F} \abs{\nabla (\varphi^4)} dv&\leq &C \int_{B_1} \abs{\varphi^{2} \nabla ^2 F} \abs{\varphi\nabla F} \abs{\nabla \varphi} dv\\
    &\leq&  \eta\int_{B_1} \varphi^4 \abs{\nabla F}^4+ \varphi^4 \abs{\nabla^2 F}^2 dv +C.
  \end{eqnarray*}
  Thanks to the boundedness of $F$, we have
  \begin{eqnarray}\label{eqn:a1}
    \int_{B_1} \varphi^4 \abs{\nabla F}^4 dv &=& \int_{B_1} \varphi^4 \langle \nabla F, \nabla F\rangle \abs{\nabla F}^2 dv \\ \nonumber
    &\leq& C\int_{B_1} \varphi^4\abs{\nabla^2 F} \abs{\nabla F}^2 dv+ C\int_{B_1} \abs{\nabla \varphi}\varphi^3 \abs{\nabla F}^3 dv\\ \nonumber
    &\leq& \frac{1}{2} \int_{B_1} \varphi^4 \abs{\nabla F}^4 +C +C\int_{B_1} \varphi^4 \abs{\nabla^2 F}^2 dv .
  \end{eqnarray}
  By taking $\alpha-1$ small, we have that $\int_{P_{3/4}} \abs{\nabla^2 F}dvdt$ is bounded, due to the boundedness of $\int_{B_1} \abs{\nabla F}^2 (\cdot,-1)dv$.

  For $k>2$, we give an indication of how the above process works. By a similar computation,
  \begin{eqnarray*}
      &&\frac{1}{2}\frac{d}{dt}\int_{B_1} \varphi \abs{\nabla^k F}^2 dv + \int_{B_1} \varphi \abs{\nabla^{k+1} F}^2 dv \\
      &\leq& C(\alpha-1) \int_{B_1} \varphi\cdot \left( \sum \prod_{i=1}^l \abs{\nabla^{a_i} F}^{b_i}\right) + \mathcal L.
  \end{eqnarray*}
  Here the summation $\sum$ is over all possible $(a_i,b_i)$ satisfying (1) $a_i=1,\cdots,k+1$, $b_i\in \mathbb N$ with $i=1,\cdots,l$ for some $l\in \mathbb N$  and (2) $\sum_{i=1}^l a_i b_i =2(k+1)$. The sum of those terms with $\sum_{i=1}^l a_i b_i<2(k+1)$ are denoted by $\mathcal L$.

  By Young's inequality, we have
  \begin{equation*}
      \int_{B_1} \varphi \sum \prod_{i=1}^l \abs{\nabla^{a_i} F}^{b_i}dv \leq C \sum_{i=1}^{k+1} \int_{B_1} \varphi \abs{\nabla^i F}^{\frac{2(k+1)}{i}} dv.
  \end{equation*}
  We now claim that for each $i=1\cdots k$, we have
  \begin{equation*}
      \int_{B_1} \varphi \abs{\nabla^i F}^{\frac{2(k+1)}{i}} dv \leq C \int_{B_1} \varphi \abs{\nabla^{i+1} F}^{\frac{2(k+1)}{i+1}} dv +C +\mathcal L.
  \end{equation*}
  The claim can be proved by induction from $i=1$, which is essentially (\ref{eqn:a1}). For $i>1$,
  \begin{eqnarray*}
      &&\int_{B_1} \varphi \abs{\nabla^i F}^{\frac{2(k+1)}{i}} dv  \\
      &\leq& C\int_{B_1} \varphi \abs{\nabla^{i-1} F} \abs{\nabla^{i+1} F} \abs{\nabla^i F}^{\frac{2(k+1)}{i}-2} dv + \mathcal L. \\
      &\leq& \eta\int_{B_1} \varphi \abs{\nabla^i F}^{\frac{2(k+1)}{i}} dv + \eta \int_{B_1} \varphi \abs{\nabla^{i-1} F}^{\frac{2(k+1)}{i-1}} dv \\
      &&+ C_\eta \int_{B_1} \varphi \abs{\nabla^{i+1} F}^{\frac{2(k+1)}{i+1}} dv +C+\mathcal L.
  \end{eqnarray*}
  By the induction assumption and choosing $\eta$ small, we see that the claim is true.
 \end{proof}

 Once we know that the $C^k$ norm of the curvature is bounded in some parabolic neighborhood, it is natural to expect a good 'gauge' in which the connection form is bounded in $C^{k+1}$. This will be the parabolic analogue of Uhlenbeck's  gauge fixing theorem. The precise statement and the proof of such a result will be interesting in its own right. For our purposes, since we have all $C^k$ bounds and the connection is a solution of a parabolic equation, we can reduce the following result to its elliptic counterpart.

 \begin{lem}\label{lem:gauge}
   Let $D(t)$ be a solution to the Yang-Mills $\alpha$-flow defined on $B\times [t_1,t_2]$. Assume that
   \begin{equation*}
     \sup_{B\times [t_1,t_2]} \abs{\nabla^k F}\leq C(k).
   \end{equation*}
   Then there is a trivialization (independent of $t$) in which $D(t)=d+A(t)$ and all derivatives (space and time) of $A(t)$ are bounded.
 \end{lem}

 \begin{proof}
     For $t=t_1$ fixed, we may apply Uhlenbeck's gauge fixing to find a trivialization such that at least all $C^k$ norms of $A(t_1)$ are bounded (see Lemma 2.3.11 in \cite{Dbook}). We can now use (\ref{eqn:flowa}) to see that $\pfrac{A}{t}$ is bounded for $B\times[t_1,t_2]$. The Newton-Leibnitz formula
     \begin{equation*}
         A(t)=A(t_1)+\int_{t_1}^{t} \pfrac{A}{t} ds
     \end{equation*}
     then implies that $A(t)$ is uniformly bounded in $M\times [t_1,t_2]$. If we take derivatives of  (\ref{eqn:flowa})  both in space and time, by noticing that the right hand side involves only $F$, we know that $\frac{\partial^k A}{\partial t^k}$ are bounded on $B\times[t_1,t_2]$. By using the Newton-Leibnitz formula again, the space derivatives of $A$ are uniformly bounded on $B\times [t_1,t_2]$. Since $A$ is bounded, one can argue inductively that both covariant derivatives and the partial derivatives are bounded.
 \end{proof}

We now prove Theorem \ref{thm:two}.
\begin{proof} Let ${A_{\alpha}}$ be the smooth solution of the Yang-Mills $\alpha$-flow in $M\times [0,\infty)$ with the same initial value $A_0$ for each $\alpha >1$.
The concentration set $\Sigma$  is defined by
\begin{equation*}
    \Sigma=\bigcap_{0<R<R_M} \left\{z\in M\times [0,\infty):\quad \liminf_{\alpha\to 1} R^{4\alpha -6}\int_{P_R(z)}(1+|F_{A_{\alpha}}|^2)^{\alpha} \,dv\,dt
\geq \varepsilon_0 \right\}
\end{equation*}
for some $\varepsilon_0>0$. It is standard to show that $\Sigma$
is closed. The same argument as in \cite{HY}  also yields that for
any two positive $t_1$ and $t_2$,   $\mathcal P^2 (\Sigma\cap
(M\times [t_1,t_2]))$ is finite, where $\mathcal P^2$ denotes the
$2$-dimensional parabolic Hausdorff measure. Moreover, for any
$t\in (0,+\infty)$, $\Sigma_t=\Sigma\cap (M\times \{t\})$ consists
of  at most finitely many points.

For a point $z_0$ outside $\Sigma$, there is a constant $R>0$ such
that for sequence of $\alpha\to 1$, we have
\[ R^{4\alpha -6}\int_{P_R(z_0)}(1+|F_{A_{\alpha}}|^2)^{\alpha} \,dv\,dt
\leq \varepsilon_0. \] Then applying Lemma \ref{lem:eregularity}, we
know that
  $F_{A_{\alpha}}$ is  uniformly  bounded  in $\alpha$ inside $P_{R/2}(z_0)$.

  Lemma \ref{lem:highorder} and Lemma \ref{lem:gauge} then imply that there is a trivialization on $P_{R/2}(z_0)$ such that $A_\alpha(t)$ is bounded in any $C^k$ norm. We then choose a sequence of such neighborhoods $\{P_i\}$ covering $M\times [0,\infty)\setminus \Sigma$. Denote the transition functions by $\sigma^\alpha_{ij}$. The $C^k$ bound  of $\sigma^\alpha_{ij}$ follows from those of $A^\alpha_i$.

  By taking a subsequence, we may assume that $\sigma^\alpha_{ij}$ converges to $\sigma_{ij}$ and $A^\alpha_i$ to $A_i$ smoothly as $\alpha$ goes to $1$. The $\sigma_{ij}$'s define a bundle $E_\infty$ over $M\times [0,\infty)\setminus \Sigma$ and the $A_i$'s define a connection $D_\infty$ of $E_\infty$. Since the convergence is strong, we know from the evolution equation of $A^\alpha_i$ that $A_i(t)$ is a solution to the Yang-Mills flow.
\end{proof}

Before we conclude this section, we would like to make some remarks. Both are related to the singular set $\Sigma$.

\begin{rem}\label{rem:firstsingular}
Let $T=\inf_{(x,t)\in \Sigma} t$ be the first concentration time in Theorem \ref{thm:two}.
  We may follow from the argument of Theorem 1.3 in \cite{HY} to show that $T$ is the same as the first singular time $T'$ of the Yang-Mills flow.
\end{rem}

  As in \cite{HY}, one may ask what more we can say about the singular set $\Sigma$. For the general case, not much is known. However, we do know something for a minimizing sequence. Precisely, we have

  \begin{prop}\label{prop:line}
     Let $D_i$ be a minimizing sequence of $YM(\cdot)$ among all smooth connections of the bundle $E$. Then we choose a subsequence of $\alpha_i\to 1$ such that $YM_{\alpha_i}(D_i)< YM(D_i)+V(M)+\frac{1}{i}$, where $V(M)$ denotes the volume of $M$. Denote by $D_i(t)$ the $\alpha_i-$flow solution with initial value $D_i$. If we consider $i\to \infty$, then the concentration set $\Sigma$ as defined above satisfies
      \begin{equation*}
          \Sigma=\bigcup_{j=1}^l \set{p_j}\times (0,\infty).
      \end{equation*}
  \end{prop}

\begin{proof}
For any $t_4>t_3>0$, since $D_i$ is a minimizing sequence, by our suitable choice of $\alpha_i\to 1$  we have
\begin{eqnarray*}
  V(M)+YM(D_i)+\frac{1}{i}&\geq & YM_{\alpha_i}(D_i)\geq YM_{\alpha_i}(D_i(t_3))\\
  &\geq &YM_{\alpha_i}(D_i(t_4))\geq V(M)+YM(D_i),\\
\end{eqnarray*}
where we have used Lemma \ref{lem:energyinequality}.

By Lemma \ref{lem:energyinequality} again, we have
\begin{equation}\label{eqn:vanish}
  \lim_{i\to \infty} \int_{t_3}^{t_4} \int_M (1+\abs{F_{D_i(t)}}^2)^{\alpha_i-1} \abs{\frac{dD_i(t)}{dt}}^2 dvdt=0.
\end{equation}
Moreover, the convergence is uniform with respect to $t_3$ and $t_4$. For any $t_2,t_1>0$,
if $(x,t_1)\notin \Sigma$, we will show $(x,t_2)\notin \Sigma$ either. Since $(x,t_1)\notin \Sigma$, we have some $r_1>0$ such that for a subsequence (for simplicity, we still denote the subsequence by $i$),
\begin{equation*}
  \int_{B_{r_1}(x)} (1+\abs{F_{D_i}(t_1)}^2)^{\alpha_i} dv \leq \frac{\varepsilon_0}{4}.
\end{equation*}
Let $\varphi$ be some cut-off function supported in $B_{r_1}(x)$.
Then
\begin{eqnarray*}
  &&\abs{\frac{d}{dt} \int_M \varphi^2 (1+\abs{F_{D_i}}^2)^{\alpha_i} dv} \\
  &=&\abs{ \int_M \alpha_i \varphi^2 (1+\abs{F_{\alpha_i}}^2)^{\alpha_i-1} \langle F_{D_i}, \pfrac{F_{D_i}}{t}\rangle dv} \\
  &\leq& \int_M \abs{\alpha_i\varphi^2 \langle D_{i}^* \left( (1+\abs{F_{D_i}}^2)^{\alpha_i-1} F_{D_i} \right), \pfrac{D_i}{t} \rangle } \\
  && + 2\alpha_i \varphi (1+\abs{F_{D_i}}^2)^{\alpha_i-1} \abs{\nabla \varphi} \abs{ F_{D_i}} \abs{ \pfrac{D_i}{t}} dv \\
  &=& \int_M \alpha_i\varphi^2(1+\abs{F_{D_i}}^2)^{\alpha_i-1} \abs{\pfrac{D_i}{t}}^2 dv \\
  && + C \left( \int_M \alpha_i\varphi^2(1+\abs{F_{D_i}}^2)^{\alpha_i-1} \abs{\pfrac{D_i}{t}}^2 dv \right)^{1/2}\\
    &&\quad\cdot  \left( \int_M \alpha_i \abs{\nabla \varphi}^2 (1+\abs{F_{D_i}}^2)^{\alpha_i-1} \abs{F_{D_i}}^2 dv \right)^{1/2}. \\
\end{eqnarray*}
The term in the last line above is bounded by a constant depending on $r_1$ but not on $i$. Therefore, if we integrate from $t_1$ to $t_3$ and let $i\to \infty$, we have, thanks to (\ref{eqn:vanish}),
\begin{equation*}
  \lim_{i\to \infty} \int_M \varphi^2 (1+\abs{F_{D_i}}^2)^{\alpha_i}(t_3) dv<\varepsilon_0/2.
\end{equation*}
Hence, by taking every $t_3\in [t_2-r_i^2,t_2+r_i^2]$, we have (for some subsequence which we labeled by $i$)
\begin{equation*}
  \lim_{i\to \infty} r_1^{4\alpha_i-6}\int_{P_{r_i}(x,t_2)} (1+\abs{F_{D_i}}^2)^{\alpha_i} dv dt \leq \varepsilon_0.
\end{equation*}
Therefore $(x,t_2)$ is not in $\Sigma$ and the proof is done.
\end{proof}

\section{Applications}
\label{sec:application}
In this section, we study the applications of the  Yang Mills $\alpha$-flow and the  Yang Mills $\alpha$-connection produced as the limit of the flow. The outline is as follows: in Subsection \ref{ssec:regularity}, we will prove the $\varepsilon$-regularity estimate for smooth Yang Mills $\alpha$-connections. In Subsection \ref{ssec:topology}, we will recall some facts about the topology of bundles and prove Theorem \ref{thm:four}.  In Subsection \ref{ssec:minimizing}, we discuss a minimizing sequence of $YM(\cdot)$ and prove Theorem \ref{thm:three}. Finally, we show how the Yang-Mills $\alpha$-flow can be used to obtain a nonminimal Yang-Mills conncetions over $S^4$.

\subsection{An $\varepsilon-$regularity lemma}
\label{ssec:regularity}
This is the analogue of what Sacks and Uhlenbeck called `main estimate'. It is necessary for the blow-up analysis. Please note that we use the $\alpha$-flow to obtain a Yang-Mills $\alpha$-connection as the limit as $t_i\to \infty$. It follows from Theorem \ref{thm:one} that the $\alpha$-connection is smooth.

    \begin{lem}
        \label{lem:mainlemma} There is $\varepsilon_1>0$ such that if $D$ is a smooth $\alpha$-Yang-Mills connection defined on $B_1$ with $\int_{B_1} \abs{F}^2 dv\leq \varepsilon_1^2$, then in some trivialization with $D=d+A$,
        \begin{equation*}
            \norm{A}_{C^k(B_{1/2})}\leq C(k)\norm{F}_{L^2(B_1)}.
        \end{equation*}
    \end{lem}

    Although we can prove it directly, we show a parabolic version, from which Lemma \ref{lem:mainlemma} follows obviously.
    \begin{thm}
      \label{thm:another} There is some $\varepsilon_1>0$ such that if $D(t)$ is a smooth solution to the $\alpha$-Yang-Mills flow on $P_1=B_1\times [-1,0]$ and
      \begin{equation*}
        \sup_{t\in [-1,0]} \int_{B_1} \abs{F}^2 dv \leq \varepsilon_1^2,
      \end{equation*}
      then
      \begin{equation*}
        \sup_{t\in [-1/4,0]} \sup_{B_{1/2}} \abs{\nabla^k F} \leq C(k).
      \end{equation*}
    \end{thm}

    The proof is omitted because it is rather well known and follows the same method as in Lemma \ref{lem:eregularity}. It suffices to use the other Bochner formula (\ref{eqn:firstbochner}). Moreover, the same method can be used to prove a stronger result by choosing a different blow-up factor. We need the following for the blow-up analysis

\begin{thm}\label{thm:better}
    There exists $\varepsilon_1>0$ such that if $D(t)$ is a smooth solution to the $\alpha$-Yang-Mills flow satisfying
    \begin{equation*}
        \sup_{[t_0-R^2,t_0]} \int_{B_{R}(x_0)} \abs{F}^2 dv\leq \varepsilon_1^2,
    \end{equation*}
    then we have
  \begin{equation*}
    \sup_{B_{ R/2}(x_0)\times [t_0-R^2/4,t_0]} \abs{F}\leq \frac{C\varepsilon^{1/2}}{ R^2},
  \end{equation*}
  where
  \begin{equation*}
  \varepsilon:=\sup_{t\in [t_0-R^2,t_0]} \int_{B_R(x_0)} \abs{F}^2 dV.
  \end{equation*}
\end{thm}
\begin{proof}
  By scaling and translation, we may assume that $R=1$, $x_0=0$ and $t_0=0$.
  Set
    \begin{equation*}
        P_r(x,t)=\{(x',t')|x'\in B_r(x) \quad \mbox{and } t-r^2\leq t'\leq t\}.
    \end{equation*}

    It is $\sup_{P_{1/2}} \abs{F}$ that we want to estimate. Find $(x_1,t_1)$ in $P_{1/2}$ such that
    \begin{equation*}
      \abs{F}(x_1,t_1)\geq \frac{1}{2}\sup_{P_{1/2}} \abs{F}.
    \end{equation*}
    It now suffices to bound $f_1:=\abs{F}(x_1,t_1)$. If we are lucky, then we have
    \begin{equation}\label{eqn:lucky}
      \sup_{P_{1/4}(x_1,t_1)} \abs{F} \leq 16 f_1.
    \end{equation}
    If not, we can find $(x_2,t_2)$ in $P_{1/4}(x_1,t_1)$ such that
    \begin{equation*}
      \abs{F}(x_2,t_2)=16 f_1.
    \end{equation*}
    By induction, we claim that after finitely many times, we have $k\in \mathbb N$, such that
    \begin{equation*}
      \abs{F}(x_k,t_k)=16^{k-1}f_1
    \end{equation*}
    and
    \begin{equation*}
      \sup_{P_{1/4^{k}}(x_k,t_k)} \abs{F}\leq 16 \abs{F}(x_k,t_k)=16^k f_1.
    \end{equation*}
    In fact, if we write $d_P$ for parabolic distance, then we have
    \begin{equation*}
      d_P( (x_k,t_k), (x_{k-1},t_{k-1}))\leq \frac{1}{4^{k-1}}.
    \end{equation*}
    Since $(x_1,t_1)$ is in $P_{1/2}$, we know $(x_k,t_k)\in P_{5/6}$ for all $k$. However, $F$ is smooth in $P_1$ and hence $\sup_{P_{5/6}}\abs{F}$ is bounded.

    We do a scaling and translation on $P_{1/4^k}(x_k,t_k)$ to get $\tilde{A}$ such that
    \begin{equation}\label{eqn:four}
      \sup_{P_{\frac{1}{4}f_1^{1/2}}} \abs{F_{\tilde{A}}}\leq 16 \quad \mbox{and } \abs{F_{\tilde{A}}}(0,0)=1
    \end{equation}
    and
    \begin{equation*}
      \sup_{[-f_1/16,0]} \int_{B_{\frac{1}{4}f_1^{1/2}}} \abs{F_{\tilde{A}}}^2 dV\leq \varepsilon.
    \end{equation*}

Applying (\ref{eqn:four}) to Theorem \ref{thm:another} and noticing Lemma \ref{lem:firstbochner}, we have
    \begin{equation*}
      \pfrac{}{t}\abs{F_{\tilde{A}}}^2 \leq \triangle \abs{F_{\tilde{A}}}^2 +C \abs{F_{\tilde{A}}}^2.
    \end{equation*}

    Consider $g(x,t)=e^{-Ct}\abs{F_{\tilde{A}}}^2$ which is a subsolution to the heat equation. By Theorem \ref{thm:another}, we know $f_1$ is bounded by a constant. Hence
    \begin{equation*}
      \int_{-f_1/16}^0 \int_{B_{\frac{1}{4}f_1^{1/2}}} g(x,t) dxdt \leq C \int_{-f_1/16}^0 \int_{B_{\frac{1}{4}f_1^{1/2}}} \abs{F_{\tilde{A}}}^2(x,t) dxdt.
    \end{equation*}
    By Mean Value inequality for linear heat equation,
    \begin{equation*}
      1=g(0,0)\leq C f_1^{-2} \varepsilon,
    \end{equation*}
    which finishes the proof of this lemma.
\end{proof}

\subsection{Connected sum of vector bundles}
\label{ssec:topology}
We recall some topological facts about vector bundles (principal bundles). Let $G$ be a connected compact Lie group. There is a topological space $BG$, which is called the classifying space of $G$, and a $G$-bundle $EG$ with $BG$ as its base, which is called the universal bundle, such that for any $G-$bundle $E$ over $M$, there is a map $f:M\to BG$ such that $E$ is just the pull back bundle $f^*(EG)$. Moreover, the isomorphism classes of $G$-bundles are in one to one correspondence with the homotopy classes of maps from $M$ to $BG$. Therefore, the classification of bundles is equivalent to the classification of continuous maps from $M$ to $BG$.

The topology of $BG$ is closely related to that of $G$. Since $EG$ is contractible, the exact sequence of homotopy groups implies that
\begin{equation*}
    \pi_{i+1} (BG)=\pi_i(G).
\end{equation*}
Moreover, it is known that for all connected Lie groups $G$, $\pi_1(G)$ is a finitely generated abelian group, $\pi_2(G)=0$ and $\pi_3(G)$ is a finitely generated free abelian group. An invariant of the classifying map $f$(hence of $E$) related to $\pi_1(G)$ is called an $\eta$ invariant. It was defined via \v{C}ech cohomology in \cite{Se}. In particular, if $\pi_1(G)=0$ or $M=S^4$, then $\eta$ is always trivial. There is another invariant called the vector Pontryagin number related to $\pi_3(G)$. For our purposes, we shall restrict ourselves to the case $M=S^4$ below. Hence, it is nothing but an element in $\pi_4(BG)=\pi_3(G)=\mathbb Z^l$.

To define the connected sum of bundles, let us consider two bundles $E_i$ over $M_i$ for $i=1,2$. Pick any $p_i\in M_i$ and let $B_i$ be a small ball around $p_i$ such that  $E_i|_{B_i}$ are trivial bundles. We obtain two manifolds with boundary $M_i\setminus B_i$ and two bundles $E_i|_{M_i\setminus B_i}$. We identify $\partial B_i$ with orientation taken into account to obtain the connected sum $M=M_1\# M_2$. Such an identification is uniquely determined topologically. We still need an identification of $E_i|_{\partial B_i}$. Although they are trivial bundles over $S^3$, there are many different bundle isomorphisms between them. Among those isomorphisms, there is a natural one. $E_i|_{\partial B_i}$ admits a trivialization inherited from the trivialization of $E_i|_{B_i}$. By identifying the two trivializations, we obtain the natural isomorphism and  a bundle $E$ over $M$, which is called the connected sum of $E_1$ and $E_2$. Since we will always consider connected manifolds $M_i$, the definition is independent of the choice of $p_i$ and the size of (small) $B_i$. We remark that $M\# S^4=M$ for any closed $4-$manifold $M$.

It is well known that when we consider the convergence of a sequence of Yang-Mills connections on bundle $E$ with bounded energy, blow-up occurs. In fact, the same discussion works for $\alpha$-Yang-Mills connections, or any other sequence of connections as long as we have the $\varepsilon$-regularity and a total energy bound. This results in a weak limit on some different bundle $E'$ and finitely many bubble connections on $E_i$ over $S^4$ for $i=1\cdots l$. The point is that $E=E'\# E_1\# \cdots \# E_l$. This follows from the removable singularity theorem of Uhlenbeck and some analysis on the neck region, which we briefly recall as follows.

Assume for simplicity that there is only one bubble. That is $A_i$, after gauge transformations, converges on $M\setminus B_\delta$ to the weak limit $A'$, and after scaling, $A_i|_{B_{\lambda_i R}}$  converges on $B_R$ to the bubble connection $\tilde{A}$. Since $\delta$ and $R$ can be arbitrary, $A'$ is defined on $M\setminus \set{p}$ and $\tilde{A}$ is defined on $\Real^4$. The removable singularity theorem claims that in fact $A'$ and $\tilde{A}$ are smooth connections of $E'$ over $M$ and $\tilde{E}$ over $S^4$. Topologically, there are different ways to extend a bundle over $M\setminus \set{p}$ to  $M$. This amounts to the choice of a trivialization of $E|_{\partial B_\delta}$ (up to topological equivalence). There is one naturally dictated by the converging sequence $A_i$. By the $\varepsilon-$regularity, if we restrict $A_i$ to $B_{\delta}\setminus B_{\delta/2}$ and scale to $B_2\setminus B_1$, it is a connection with arbitrarily small curvature (in any norm). This decides a trivialization (see Lemma 2.4 in \cite{remove}). Similar analysis works for the bubble connection on $B_{2\lambda_i R}\setminus B_{\lambda_i R}$.

To see that $E$ is the connected sum of $E'$ and $\tilde{E}$, it suffices to show that the trivializations of $E$ on $B_{\delta}\setminus B_{\delta/2}$ and $B_{2\lambda_i R}\setminus B_{\lambda_i R}$ agree with each other. This is related to how the bubble tree is constructed. If one follows the process of Ding and Tian \cite{DT}, we know that the energy of the $A_i$ restricted to $B_t\setminus B_{t/2}$ are smaller than any given $\varepsilon_1$ for $t\in [2\lambda_i R, \delta]$. For each $t$, the smallness of energy and $\varepsilon-$regularity implies a choice of trivialization. As $t$ changes from $2\lambda_R$ to $\delta$, we see that the two trivializations can be continuously deformed to each other. If one follows the construction of Parker \cite{Parker}, we have the total energy over the neck region $B_\delta\setminus B_{2\lambda_i R}$ is small, say smaller than $\varepsilon_1$. Using the trivialization over $B_{2\lambda_i R}\setminus B_R$, we may extend the connection to $B_R$ with a controlled amount of the energy. We can do the same  at the infinity to obtain a smooth connection over $S^4$ whose energy is smaller than a multiple of $\varepsilon_1$. Hence, the bundle must be trivial and it implies that the two trivializations agree with each other.

We now prove Theorem \ref{thm:four}.
\begin{proof}[Proof of Theorem \ref{thm:four}]
  Recall that $G$-bundles over $S^4$ correspond to the homotopy classes of maps from $S^4$ to the classifying space $BG$ of $G$, and that $\pi_4(BG)=\pi_3(G)$. Assume the theorem is not true. Then there are at most $r-1$ $G$-bundles which admit  Yang-Mills $G-$connections. Let $a_1,\cdots,a_{r-1}$ be elements in $\pi_4(BG)$ corresponding to these $G$-bundles. By our assumption, there is $a\in \pi_4(BG)$  which is not generated by $\{a_1,\cdots,a_{r-1}\}$.

  Let $E$ be the bundle corresponding to $a$. Pick any smooth connection on $E$. Consider the $\alpha$-flow starting from it. Theorem \ref{thm:one} gives a  Yang-Mills $\alpha$-connection $A_\alpha$ for each $\alpha>1$. Since $E$ is not a trivial bundle and $S^4$ is simply connected, $A_\alpha$ cannot be flat. Take the limit as $\alpha$ to $1$.

  If the convergence is strong, then we find a Yang-Mills $G$-connection, which contradicts the  choice of $a$. If not, the bundle $E$ splits into a connected sum of $E_1$,\ldots,$E_l$ over $S^4$, and each admits a Yang-Mills $G$-connection, which also contradicts the choice of $a$.
\end{proof}

\subsection{Minimizing sequences of $YM(\cdot)$}
\label{ssec:minimizing}

In this subsection, we prove Theorem \ref{thm:three}. For a closed $4-$manifold $M$ and the $G$-bundle $E$, let $m(E)$ be the infimum of $YM(A)$ for all $G-$connections $A$ of $E$.

First, let us show a general result which has nothing to do with the blow-up.
\begin{prop}\label{prop:always}
 If $E=E'\# E_1\#\cdots \#E_l$, where $E'$ is a bundle over $M$ and $E_i$ are bundles over $S^4$, then
\begin{equation*}
  m(E)\leq m(E')+\sum_{i=1}^l m(E_i).
\end{equation*}
\end{prop}

\begin{proof}
  For simplicity, consider $l=1$. If suffices to show that for any $\varepsilon>0$ and any two connections $D_1$ and $D_2$ of $E'$ and $E_1$ respectively, we may construct a connection $D$ of $E$ such that
  \begin{equation*}
    YM(D)\leq YM(D_1)+YM(D_2)+\varepsilon.
  \end{equation*}
  (This is exactly Lemma 5.7 in \cite{Isobe}).  For completeness, we also give a proof here.

  Given any smooth connection $D_i$ and a trivialization of the bundle over some ball $B$, by multiplying  by a cut-off function, we may assume that $D_i$ is flat in a smaller ball at the expense of any small change of the energy. More precisely, for any $\varepsilon>0$, there is a $\delta>0$ and we have another connection $D_i'$ such that

   (1) $D_i=D_i'$ outside $B_\delta$;

    (2) $D_i'=d$ on $B_{\delta/2}$;

    (3) $\abs{YM(D'_i)-YM(D_i)}< \varepsilon$.

    Indeed, if $D_i=d+A_i$ on $B$, due to the smoothness of $A_i$, there exists $\delta>0$ such that if we scale $B_\delta$ to $B_2$, $D_i$ becomes $d+\tilde{A}_i$ with $\norm{\tilde{A}_i}_{C^k}$ as small as we need.

  Let $\varphi$ be a cut-off function: $\varphi\equiv 1$ on $B_2\setminus B_{3/2}$ and $\varphi\equiv 0$ in $B_{1}$. Consider a new connection $d+(\varphi \tilde{A}_i)$. It agrees with $d+\tilde{A}_i$ outside $B_{3/2}$ and is $d$ in $B_{1}$. We scale $d+(\varphi \tilde{A}_i)$ back to $B_\sigma$ and denote the new connection by $D_i'$. It remains to see that the change in the energy is small. Due to the scaling invariance of energy, it suffices to check that any $C^k$ norm of $F=d(\varphi \tilde{A}_i)+ [\varphi \tilde{A}_i,\varphi \tilde{A}_i]$ is small on $B_2$.

  Fix $p\in M$ and $q\in S^4$. By the above construction, we may assume that in $B_\delta(p)$ and $B_{\delta}(q)$, there is a trivialization such that the connection is just $d$. Via the stereographic projection, $D_1$ is a connection over $\Real^4$, which outside $B_R$ is nothing but $d$ in some trivialization. We further scale it down to assume that $R=\delta/2$. We can now obtain a new connection by gluing $D'$ on $M\setminus B_{\delta/2}$ and $D_1$ on $B_R$. Since there is no energy at all in the overlap domain, the lemma is proved.
\end{proof}

We then consider a minimizing sequence. For a given bundle $E$, let $D_i$ be a minimizing sequence with
\begin{equation*}
  \lim_{i\to \infty} YM(D_i)=m(E).
\end{equation*}
Since $D_i$ is smooth, we can find $\alpha_i$ close to $1$ such that
\begin{equation*}
  YM_{\alpha_i}(D_i)\leq YM(D_i)+V(M)+\frac{1}{i}.
\end{equation*}
Let $D_i(t)$ be the solution of the $\alpha_i$-Yang-Mills flow from $D_i$ and set $D_i'=D_i(1)$. Then,
\begin{equation*}
  YM(D_i')+V(M)\leq YM_{\alpha_i}(D_i')\leq YM(D_i)+V(M)+\frac{1}{i}.
\end{equation*}
This implies that $D_i'$ is another minimizing sequence.

In order to do the blow-up analysis for $D_i'$, we need the following $\varepsilon-$regularity result,
\begin{lem}
    There exists $\varepsilon>0$ such that if $B_r(x)\subset M$ satisfies
    \begin{equation*}
        \lim_{i\to \infty} \int_{B_r(x)} \abs{F_{D_i'}}^2 dv \leq \varepsilon,
    \end{equation*}
    then
    \begin{equation*}
        \norm{\nabla_{D_i'}^k F_{D_i'}}_{C^0(B_{r/2}(x))} \leq C r^{-k-2}.
    \end{equation*}
\end{lem}

\begin{proof}
    The proof relies on Theorem \ref{thm:another} and $\varepsilon$ will be determined by $\varepsilon_1$ and the energy bound for our minimizing sequence.

By our choice of $\alpha_i$, we have
\begin{equation*}
    \lim_{i\to \infty} \int_M \left(1+\abs{F_{D_i'}}^2\right)^{\alpha_i} -\left(1+\abs{F_{D_i'}}^2\right) dv =0.
\end{equation*}
Hence, for $i$ sufficiently large,
\begin{equation*}
    \lim_{i\to \infty} \int_{B_{r}(x)} (1+\abs{F_{D_i'}}^2)^{\alpha_i}-1\, dv \leq 2 \varepsilon.
\end{equation*}
The local energy inequality (Lemma \ref{lem:localenergy}) implies that there exists $\sigma>0$ depending on the total energy and $\varepsilon$ such that for $i$ sufficiently large,
\begin{equation*}
    \sup_{t\in [1-\sigma r^2,1]} \int_{B_{r}(x)}(1+\abs{F_{D_i'}}^2)^{\alpha_i}-1\, dv \leq 3 \varepsilon.
\end{equation*}
Therefore,
\begin{equation*}
    \sup_{t\in [1-\sigma r^2,1]} \int_{B_{r}(x)}\abs{F_{D_i'}}^2 \, dv \leq 4 \varepsilon.
\end{equation*}
Set $\varepsilon=\frac{1}{4}\varepsilon_1$ and the proof follows from Lemma \ref{thm:another}.
\end{proof}

Now we can do the well-known blow-up analysis for $D_i'$. If there are nontrivial bubbles and $E=E'\# E_1\#\cdots \# E_l$, then
\begin{equation*}
  m(E)=\lim_{i\to\infty} YM(D_i')\geq m(E')+\sum_{i=1}^l m(E_i).
\end{equation*}
This together with Proposition \ref{prop:always} will imply the energy identity:
\begin{prop}
    Let $D_i$ be a minimizing sequence of the Yang-Mills functional among all smooth connections of the bundle $E$ over $M$. Then, there exist bundles $E'$ over $M$ and $E_1,\cdots, E_l$ over $S^4$ for some $l\geq 0$ and Yang-Mills connections $D_{\infty}'$ and $\tilde{D}_1,\cdots, \tilde{D}_l$ such that
  \begin{equation*}
      \lim_{i\to\infty} YM(D_i)=YM(D'_{\infty})+\sum_{i=1}^l YM(\tilde{D}_i).
  \end{equation*}
\end{prop}

Next, it remains to study the relation between the limit connection $D_{\infty}'$ and the weak limit $D_{\infty}$ of Sedlacek \cite{Se}.

We try to prove that the two limit (two Yang-Mills connection on two smooth bundles) are globally the same up to gauge transformations. This is the best one could hope for.

Let $S$ be the union of energy concentration sets, both for $D_i$ in the Sedlacek limit and for $D_i'$ above. Let $\set{U^\beta}$ be an open cover of $M\setminus S$. We shall consider three bundles.

(1) The original one where the minimizing sequences and their $\alpha$-flow lies on is denoted by $E$.

(2) The weak limit bundle, $E_1$, where the weak limit of $D_i$ lies. In the paper of Sedlacek, it is given by transition functions. However, it is convenient to think  of it as an abstract bundle, with a set of trivialization.

(3) The strong limit bundle, $E_2$, where the weak limit $D_i'$ lies.

The convergence of the minimizing sequence $D_i$ on $E$ in \cite{Se} can be reformulated as follows. For each $D_i$, there is a trivialization $e_i^\beta$ in which $D_i=d+A_i^\beta$, where $\norm{A_i^\beta}_{W^{1,2}}$ is bounded. $g_i^{\beta\gamma}$ will denote the transition functions. There is a trivialization $e^\beta$ of $E_1$ when restricted to $M\setminus S$, in which the weak limit $D_\infty=d+A_\infty^\beta$. We denote the transition functions by $g^{\beta\gamma}$. We know
\begin{equation*}
    \norm{A_i^\beta-A_\infty^\beta}_{W^{1,2}}\to 0.
\end{equation*}
\begin{rem}
    This convergence was shown to be weakly $W^{1,2}$ in \cite{Se} and was shown to be strong by Isobe in \cite{Isobe}.
\end{rem}
There is a bundle map $\varphi_i^\beta:E|_{U^\beta}\to (E_1)|_{U^\beta}$ by identifying trivialization $e_i^\beta$ and $e^\beta$. The above convergence can be written as
\begin{equation}\label{eqn:leftcon}
    \norm{(\varphi_i^\beta)^* D_\infty-D_i}_{W^{1,2}(U^\beta)}\to 0.
\end{equation}
In \cite{Se}, $\varphi_i^\beta$ and $\varphi_i^\gamma$ cannot be fitted together to get a larger bundle map. However, we have the following relation between them.

Let $v$ be any vector of $E|_{U^\beta\cap U^\gamma}$. Suppose that
\begin{equation*}
    v=\tilde{v} e^\beta_i = g^{\beta\gamma}_i \tilde{v} e^\gamma_i.
\end{equation*}
By definition,
\begin{equation}\label{eqn:good}
    \varphi^\beta_i(v)= \tilde{v} e^\beta = g^{\beta\gamma} \tilde{v} e^\gamma = \varphi^\gamma_i (g^{\beta \gamma} g^{\gamma\beta}_i v e^\beta_i)= g^{\beta\gamma}g^{\gamma\beta}_i \varphi^\gamma_i(v).
\end{equation}
The relation (\ref{eqn:good}) will be important for us later.

Next, we describe the strong convergence of $D_i'$ to $D'_\infty$. We know there is a sequence of bundle maps $\sigma_i$ from $E|_{M\setminus S}$ to $E_2|_{M\setminus S}$ such that
\begin{equation*}
    \norm{\sigma_i^* D_\infty'- D_i'}_{C^{k}(K)}\to 0
\end{equation*}
for any compact $K$ in $M\setminus S$.
For any $\beta$, we have
\begin{equation}
    \norm{\sigma_i^* D_\infty'- D_i'}_{C^{k}(U^\beta)}\to 0.
    \label{eqn:rightcon}
\end{equation}
By our construction, we know
\begin{equation*}
    \norm{D_i-D'_i}_{L^2}\to 0.
\end{equation*}
Hence,
\begin{equation*}
    \norm{(\varphi^\beta_i)^* D_\infty- \sigma_i^* D'_\infty}_{L^2(U^\beta)}\to 0.
\end{equation*}
That is
\begin{equation}\label{eqn:local}
    \norm{D_\infty- (\eta^\beta_i)^* D'_\infty}_{L^2(U^\beta)}\to 0,
\end{equation}
where $\eta_i^\beta= \sigma_i \circ (\varphi^\beta_i)^{-1}$ is a bundle map from $E_1|_{U^\beta}$ to $E_2|_{U^\beta}$.

We claim that $\eta^\beta_i$ converges to $\eta^\beta$ in weak $W^{1,2}$ topology and $D_\infty= (\eta^\beta)^* D'_\infty$ on $U^\beta$. To see this, consider the meaning of (\ref{eqn:local}) in trivialization $e^\beta$ and $f^\beta$. (Here $f^\beta$ is a trivialization of $E_2$ on $U^\beta$.) Since $D_\infty=d+A_\infty$ and $D'_\infty=d+A'_\infty$, we have
\begin{equation*}
    \norm{A_\infty-(s^{-1}ds + s^{-1}A'_\infty s)}_{L^2(U^\beta)}\leq C.
\end{equation*}
Here the $s$ is the map $\eta^\beta_i$ in a trivialization and it is bounded in $W^{1,2}$. Hence  our claim follows. Moreover, although the convergence is only weakly $W^{1,2}$, $\eta^\beta$ is smooth since it maps smooth connections to smooth connections.

We next claim that $\eta^\beta$ and $\eta^\gamma$ agree  over $U^\beta\cap U^\gamma$. Hence, this gives a global bundle map $\eta$ from $E_1|_{M\setminus S}$ to $E_2|_{M\setminus S}$. To see this, it suffices to check that
\begin{equation*}
    \lim_{i\to \infty} \sigma_i\circ (\varphi_i^\beta)^{-1}=\lim_{i\to \infty} \sigma_i\circ (\varphi_i^\gamma)^{-1}.
\end{equation*}
Due to the smoothness of $\eta^\beta$ and $\eta^\gamma$, it suffices to check the above for a dense set of $x\in U^\beta\cap U^\gamma$. Thanks to (\ref{eqn:good}) and the $W^{1,2}$ weak convergence of $g^{\beta\gamma}_i$ to $g^{\beta\gamma}$, we have a dense set $W$ such that for $x\in W$ and any $v\in (E_1)_x$, we have
\begin{equation*}
    (\varphi_i^\beta)^{-1}(v)-(\varphi_i^\gamma)^{-1}(v)\to 0.
\end{equation*}
Because $\sigma_i$ is a linear map and $\sigma_i$ lies in $G\subset SO(r)$ ($r$ is the rank of $E$), we have
\begin{equation*}
    \lim_{i\to \infty} \sigma_i \circ (\varphi_i^\beta)^{-1} (v)- \sigma_i \circ (\varphi^\gamma_i)^{-1}(v)=0.
\end{equation*}

Now we have a bundle map $\eta$ defined  on $M\setminus S$ satisfying $\eta^* D'_\infty=D_\infty$. Finally, since $D_\infty$ and $D'_\infty$ are smooth connections, $\eta$ extends automatically to  a global smooth gauge transformation with $\eta^* D'_\infty=D_\infty$. In fact, locally on $B\setminus \set{0}$,
\begin{equation*}
    A_\infty= \eta^{-1} d\eta +\eta^{-1} A'_\infty \eta,
\end{equation*}
which implies $\eta$ and all its derivatives are bounded on $B\setminus \set{0}$ since $A_\infty$ and $A'_\infty$ are smooth over $B$.

Hence, we finish the proof of Theorem \ref{thm:three}.

\subsection{Another approach for Min-Max of the Yang-Mills functional}
\label{ssec:ssu}

It is well known that the Yang-Mills functional in dimension 4 does not satisfy the Palais-Smale condition, which caused great difficulty in applying   Morse theory to show the existence of a nonminimal critical point. In 1989, Sibner, Sibner and Uhlenbeck \cite{SSU} proved the existence of nonminimal Yang-Mills connections on the trivial $SU(2)$ bundle over $S^4$. They used the fundamental relationship between $m-$equivariant gauge fields on $S^4$ and monopoles on hyperbolic $3-$space $\mathbb H^3$ as presented by Atiyah \cite{A}. If we identify $S^4$ with $\Real^4\cup \set{\infty}$ by stereographic projection, we may introduce the following coordinates
\begin{equation*}
    (z,\theta,(x,y))\mapsto (z\cos \theta,z\sin \theta,x,y)\in \Real^4.
\end{equation*}
Hence, one can define a $U(1)$ action on $S^4$ by
\begin{equation*}
    q(\theta')(z,\theta,(x,y))=(z,\theta+\theta' (\mbox{mod} 2\pi),(x,y))
\end{equation*}
and leaving other points in $S^4$ not represented by this coordinate system fixed.

Let $\set{\hat{i},\hat{j},\hat{k}}$ be a standard basis for $\mathfrak {su}(2)$ and $s(\theta)=e^{\hat{i}m\theta}(m\geq 2)$ be a homeomorphism from $U(1)$ to $SU(2)$. A connection $D$ is called {\it an $m$-equivariant connection} if
\begin{equation*}
    q(\theta)^* D= s(\theta)^{-1}\circ D\circ s(\theta)
\end{equation*}
for all $\theta\in U(1)$. Denote the set of all $m$-equivariant connections of the trivial $SU(2)$ bundle over $S^4$ by $\mathcal M$.

The authors of \cite{SSU} followed a construction of Taubes \cite{loop} to find a non-contractible loop of connections $D^\gamma (\gamma\in S^1)$ of $m$-equivariant connections in $\mathcal M$, satisfying
\begin{equation}
    YM(D^\gamma)<8\pi m.
    \label{eqn:smallym}
\end{equation}
The connections in Lemma 2 of \cite{SSU} are in $W^{1,\infty}$, but by approximation, we can assume that they are smooth and (\ref{eqn:smallym}) remains true. Since they are smooth, we know
\begin{equation*}
    YM_\alpha(D^\gamma)<8\pi m+\omega_4
\end{equation*}
for sufficiently small $\alpha$. Here $\omega_4$ is the volume of $S^4$.

We can now apply  the Yang-Mills  $\alpha$-flow to the loop. The $\alpha$-flow preserves symmetry, so that the flow stays in $\mathcal M$. By Theorem \ref{thm:stable}, we obtain a deformation of the circle in $\mathcal M$. We then claim that we obtain a nontrivial  Yang-Mills  $\alpha$-connection $D_\alpha$ with $YM_\alpha(D_\alpha)<8\pi m+\omega_4$. Otherwise, the flow will converge  to the flat connection for any $\gamma\in S^1$, which will result in a contraction of the loop to a single point in $\mathcal M$. This is not possible.

The energy of these  Yang-Mills  $\alpha$-connections $D_\alpha$ has a uniform lower bound. This is a generalized gap theorem similar to the result of Bourguignon and Lawson \cite{BL}.

\begin{lem}
    \label{lem:gap}
    There is $\kappa>0$ depending only on $G$ such that any nontrivial  Yang-Mills  $\alpha$-connection $D_\alpha$ on $S^4$ satisfies
    \begin{equation*}
        YM(D_\alpha)>\kappa.
    \end{equation*}
\end{lem}

\begin{proof}
    Recall that we have proved a stronger Bochner formula (\ref{eqn:betterbochner}) than stated in Lemma \ref{lem:firstbochner}. For our purpose here, $\partial_t \abs{F}^2$ vanishes and the $Ric\wedge g+2R$ is just the $4$ times of the identify map on $2-$forms. Hence,
  \begin{equation*}
      -\nabla_{e_i} \left( (\delta_{ij}+b_{ij})\nabla_{e_j}\abs{F}^2 \right) \leq C \abs{F}^3-3\abs{F}^2,
  \end{equation*}
when $\alpha-1$ is small.
Multiplying both sides by $\abs{F}^2$ and integrating over $S^4$, we have
\begin{equation*}
    \int_{S^4} \abs{\nabla \abs{F}^2}^2 +\abs{F}^4 \leq C \int_{S^4} \abs{F}^5.
\end{equation*}
By the Sobolev inequality and the H\"older inequality, we obtain
\begin{equation*}
    \left( \int_{S^4} \abs{F}^8 \right)^{1/2}\leq C \left( \int_{S^4} \abs{F}^2 \right)^{1/2} \left( \int_{S^4} \abs{F}^8 \right)^{1/2}.
\end{equation*}
This implies that $F$ is identically zero if the energy is small.
\end{proof}

Now, we may pass to the limit $\alpha\to 1$. Note that $\kappa<YM(D_\alpha)<8\pi m$. The rest of the proof goes just like Theorem 1 in \cite{SSU}. If the convergence of $D_\alpha$ is strong, we obtain  a nonminimal Yang-Mills connection on the trivial $SU(2)$ bundle over $S^4$. If not, the energy bound $8\pi m$ implies that either the weak limit  or one of the bubbles is a nontrivial Yang-Mills connection on the trivial $SU(2)$ bundle (hence nonminimal), because the energy is not enough for two nontrivial bundles.

\begin{acknowledgement}
 {The research  of the
first author was supported by the Australian Research Council
grant.  A part of the work was done when
Tian and Yin visited the University of Queensland in 2012.}
\end{acknowledgement}

\end{document}